\documentclass{amsart}
\usepackage{amssymb}
\usepackage[mathscr]{eucal}
\usepackage{xspace}
\usepackage[latin1]{inputenc}

\usepackage{graphicx}

\newcommand{\mm}{\mathfrak{m}}

\newcommand{\R}{\mathbb{R}}

\newcommand{\D}{\Delta}
\newcommand{\Z}{\mathbb{Z}}
\newcommand{\N}{\mathbb{N}}

\newcommand{\GG}{\mathcal{G}}

\newcommand{\KK}{\mathcal{K}}
\renewcommand{\H}{\mathcal{H}}

\newcommand{\E}{{\mathcal E}}

\newcommand{\gD}{\Delta}

\newcommand{\po}{\partial}

\newcommand{\ve}{\varepsilon}

\newcommand{\la}{\langle}
\newcommand{\ra}{\rangle}

\newcommand{\loc}{{\text{loc}}}
\newcommand{\X}{\times}

\renewcommand{\d}{\delta}
\renewcommand{\l}{\lambda}

\renewcommand{\a}{\alpha}
\renewcommand{\b}{\beta}

\newcommand{\s}{\sigma}
\newcommand{\g}{\gamma}

\newcommand{\z}{\zeta}

\newcommand{\Om}{\Omega}

\newcommand{\supp}{\text{\rm supp}\,}
\newcommand{\re}{\mathbb{R}}
\newcommand{\M}{{\mathcal M}}
\newcommand{\V}{{\mathcal V}}

\newcommand{\FSM}{\operatorname{FS}}
\newcommand{\FS}{\operatorname{FS}}

\renewcommand{\div}{\text{\rm div}\,}

\newcommand{\Lip}{\text{\rm Lip}}

\renewcommand{\subset}{\subseteq}

\newcommand{\K}{\mathcal{K}}

\newcommand{\BV}{\operatorname{BV}}

\newcommand{\BUC}{\operatorname{BUC}}

\newcommand{\AP}{\operatorname{AP}}
\newcommand{\PAP}{\operatorname{PAP}}

\newcommand{\dist}{\operatorname{dist}}

\newcommand{\WAP}{\operatorname{WAP}}
\newcommand{\util}[1]{\underline{#1}}

\newcommand{\BB}{{\mathcal B}}
\renewcommand{\AA}{{\mathcal A}}

\newcommand{\CK}{{\mathcal K}}

\newcommand{\medint}{{\mbox{\vrule height3.5pt depth-2.8pt
          width4pt}\mkern-13mu\int\nolimits}}
\newcommand{\Medint}{\mkern12mu\mbox{\vrule height4pt
         depth-3.2pt
          width5pt}\mkern-16.5mu\int\nolimits}

\newcommand{\sgn}{\operatorname{sgn}}
\renewcommand{\supp}{\operatorname{supp}}

\renewcommand{\epsilon}{\varepsilon}

\theoremstyle{plain}
\newtheorem{theorem}{Theorem}[section]
\newtheorem{corollary}{Corollary}[section]

\newtheorem{lemma}{Lemma}[section]
\newtheorem{proposition}{Proposition}[section]
\theoremstyle{definition}
\newtheorem{definition}{Definition}[section]
\theoremstyle{remark}

\newtheorem{remark}{Remark}[section]
\numberwithin{equation}{section}

\begin{document}

\title[Homogenization of Degenerate Porous Medium Type Equations]
{Homogenization  of  Degenerate Porous Medium Type Equations\\ in Ergodic Algebras}
\author{Hermano Frid}
\address{Instituto de Matem\'atica Pura e Aplicada - IMPA\\ Estrada Dona Castorina, 110\\
Rio de Janeiro, RJ, 22460-320, Brazil}
\email{hermano@impa.br}
\author{Jean Silva}
\address{Instituto de Matem\' atica\\
Universidade Federal do Rio de Janeiro\\
P.O. Box 68530, CEP 21945-970, Rio de Janeiro, RJ, Brazil}
\email{jean@im.ufrj.br}

\date{}

\begin{abstract} 
We consider the homogenization problem for  general porous medium type equations of the form $u_t=\D f(x,\frac{x}{\ve}, u)$. The pressure function $f(x,y,\cdot)$ may be of two different types. In the type~1 case, $f(x,y,\cdot)$ is a general  strictly increasing function; this is a mildly degenerate case. In the type~2 case,  $f(x,y,\cdot)$ has the form $h(x,y)F(u)+S(x,y)$, where $F(u)$ is just a nondecreasing function;
this is a strongly degenerate case.  We address the initial-boundary value problem for a general,  bounded or unbounded, domain $\Om$, with null (or, more generally, steady) pressure condition on the boundary.   The homogenization is carried out in the general context of ergodic algebras.  As far as the authors know, homogenization of  such degenerate quasilinear parabolic equations is addressed here for the first time.  We also review the existence and stability theory for such equations and establish new results  needed for the homogenization analysis. Further, we include some new results on algebras with mean value, specially a new criterion  establishing the null measure of level sets of elements of the algebra, which is useful in connection with the homogenization of porous medium type equations in the type~2 case.  
\end{abstract}

\maketitle

\section{Introduction} \label{S:0}

In this paper we consider the homogenization of a porous medium type equation of the general form
\begin{equation}\label{eI.01}
u_t=\D f(x,\frac{x}{\ve},  u),
\end{equation}
with $(x,y,t)\in\Om\X\R^n\X(0,\infty)$, and $\Om\subset\R^n$ is a, bounded or unbounded, open set . Here   $f$ is a continuous function of $(x,y,u)$ and $f(x,y,\cdot)$ is locally Lipschitz continuous, uniformly in $(x,y)$, and may be of two different types: 
\begin{itemize}
\item In the type~1 case, $f(x,y,\cdot)$ is a  general strictly increasing function; this is a mildly degenerate case. 

\item In the type~2 case,  $f(x,y,u)$ has the form $h(x,y)F(u)+S(x,y)$, where $F(u)$ is just a nondecreasing function, which is {\em not} strictly increasing;
this is a strongly degenerate case.  Let us denote by $G$ the strictly increasing right-continuous function such that $F(G(v))=v$, for all $v\in\R$.
\end{itemize}

We consider the initial-boundary value problem where we prescribe an initial condition of the form 
\begin{equation}\label{eI.02}
u(x,0)=u_0(x,\frac{x}{\ve}),
\end{equation} 
and a boundary condition of the form  
\begin{equation}\label{eI.03}
f(x,\frac{x}{\ve},u(x,t))\, |\, \po\Om\X(0,\infty)=0. 
\end{equation}
Our analysis applies equally well to a more general, non-homogeneous, boundary condition of the form  
$$
f(x,\frac{x}{\ve}, u(x,t))\,|\,\po\Om\X(0,\infty)=\b(x), 
$$
for a function $\b\in C(\bar\Om)\cap H_\loc^1(\bar\Om)$, where $C\bar \Om)$ denotes the space of bounded continuous functions on $\bar \Om$, and $H_\loc^1(\bar \Om)$ denotes the space of functions defined on $\bar \Om$, which multiplied by any function in $C_c^\infty(\R^n)$ gives a function in the usual Sobolev (Hilbert) space of first order $H^1(\bar\Om)$. We address the homogeneous case \eqref{eI.03} just for conveniency. 

For fixed $(x,u)\in\Om\X\R$, we will assume that $f(x,\cdot,u)\in \AA(\R^n)$, where $\AA(\R^n)$ is a general ergodic algebra, which means an algebra with mean value that is ergodic. An algebra with mean value (algebra w.m.v., for short) is an algebra of bounded uniformly continuous functions on $\R^n$, invariant by translations, each member of which possesses a mean value. It is said to be ergodic, roughly speaking, if, for any function $\varphi\in\AA$, the averages of the translates $\varphi(\cdot+y)$, in balls of radius $R>0$, converge as $R\to\infty$, in the norm of the mean value of the square of the absolute value, to the mean value $\bar\varphi$ of $\varphi$. The most elementary example of an ergodic algebra is the space of continuous functions in $\R^n$, which are periodic in each coordinate $\varphi(x+\tau_i e_i)$, $i=1,\cdots,n$, for certain constants $\tau_i\in\R$, where $e_i$ are the elements of the canonical basis of $\R^n$. Another well known example is the space of almost periodic functions which may be defined as   the closure in the $\sup$ norm of the space spanned by the set $\{e^{i\l\cdot x}\,:\,\l\in\R^n\}$, in the complex case, or the real parts of the functions in such space, in the real case ({\em cf.} \cite{Bo,B}). Many other examples are known such as the space of Fourier-Stieltjes transforms, the weakly almost periodic functions, etc.; we will comment a bit on such examples in Section~\ref{S:1}, below.  We recall that the theory of algebras w.m.v.\ and ergodic algebras was first developed by Zhikov and Krivenko in \cite{ZK} (see also \cite{JKO}). Concerning the initial data in \eqref{eI.02}, we will, in general,  assume that  $u_0\in L^\infty(\Om;\AA(\R^n))$.

We present two main results concerning the homogenization of the initial-boundary value problem \eqref{eI.01}, \eqref{eI.02},\eqref{eI.03}. 

Our first main result applies to an unbounded domain $\Om$ and a general ergodic algebra $\AA(\R^n)$, but we have to restrict  ourselves to initial data that are ``well-prepared'', that is, of the form
$$
u_0(x,y)=g(x,y, \phi_0(x)),
$$
for some $\phi_0\in L^\infty(\Om)$, where, for all $(x,y)\in\Om\X(0,\infty)$,  $g(x,y,\cdot)$ is the strictly increasing right-continuous function satisfying $f(x,y,g(x,y,v))=v$, for all $v\in\R$.

Our second main result applies to a general initial data $u_0\in L^\infty(\Om;\AA(\R^n))$, but we have to compromise  restricting ourselves to a bounded domain $\Om$ and to an ergodic algebra $\AA(\R^n)$ which is a regular algebra w.m.v., examples of the latter being provided by the periodic, almost periodic, and Fourier-Stieltjes transform functions, the precise definition  being left to Section~\ref{S:1}.  

Both main results ({\em cf.} Theorems~\ref{T:6.1} and~\ref{T:7.1}, in Sections~\ref{S:5} and~\ref{S:6}, respectively)  establish, under the mentioned assumptions, the weak star convergence in 
$L^\infty(\Om\X(0,\infty))$ of the entropy solutions $u_\ve(x,t)$ of  \eqref{eI.01},\eqref{eI.02},\eqref{eI.03} (see, Definition~\ref{D2''}) to the entropy solution $\bar u(x,t)$ of the problem
\begin{equation}\label{eI.04}
\begin{aligned}
&u_t=\D \bar f(x,u),\\
&u(x,0)=\bar u_0(x),\\
&\bar f(x,u(x,t))\,|\,\po\Om\X(0,\infty)=0,
\end{aligned}
\end{equation}
where 
$$
\bar u_0(x)=\Medint_{\R^n} u_0(x,y)\,dy,
$$
and $\bar f(x,u)$ is defined by $\bar f(x,\bar g(x,v))=v$, with
\begin{equation}\label{eI.05}
\bar g(x,v):=\Medint_{\R^n}g(x,y,v)\,dy.
\end{equation}
In the case where $f$ is of type~2, $\bar g(x,\cdot)$, defined by  \eqref{eI.05}, may, in general, be discontinuous, which is a bad situation for defining precisely $\bar f(x,\cdot)$, only from the knowledge of $\bar g$. In order to avoid such indetermination, we impose the additional assumption, concerning the functions $h(x,y)$ and $S(x,y)$ appearing in the definition of a pressure function of type~2:
\begin{equation}\label{eI.06}
\mm\left(\{z\in\KK\,:\,\a h(x,z)+S(x, z)=v\}\right)=0,
\end{equation}
for all $(x,v)\in\Om\X\R$,  with $\a$ belonging to the discontinuity set of $G$, where $\KK$ is the compact space associated with the algebra w.m.v.\ $\AA(\R^n)$ 
 and $\mm$ is the corresponding probability measure in $\KK$ (see, Theorem~\ref{T1} below, established in \cite{AFS}). Under the assumption \eqref{eI.06} the function $\bar g(x,\cdot)$, defined in \eqref{eI.05}, turns out to be continuous and strictly increasing, and so is its inverse $\bar f(x,\cdot)$,  which means that the limit problem \eqref{eI.04} has, in any case, a pressure function of type~1. This has the additional advantage of making much easier to check the uniqueness of the solution of the limit problem, since, for pressure functions of type~1, the notions of entropy and weak solutions coincide ({\em cf.} Definition~\ref{D2''}).   

Moreover, both Theorems~\ref{T:6.1} and~\ref{T:7.1} also give the existence of correctors, that is, both theorems assert that
\begin{equation}\label{eI.07}
u_\ve(x,t)-g(x,\frac{x}{\ve}, \bar f(\bar u(x,t)))\to 0, \quad \text{as $\ve\to0$ in $L_\loc^1(\Om\X(0,\infty)$}.
\end{equation}
Again, to obtain \eqref{eI.07} in the case where $f$ is of type~2, we make essencial use of \eqref{eI.06}, which makes a study of sufficient conditions to guarantee this null measure property of great interest, and we obtain such conditions herein ({\em cf.} Lemma~\ref{L:zeromeasure}), as we will comment below. 

Before giving a brief idea of the techniques used to obtain the main homogenization results, Theorems~\ref{T:6.1} and~\ref{T:7.1}, we make a  further brief comment about assumption~\ref{eI.06}.  This assumption leads us to the question about necessary conditions for the vanishing of the measure of  level sets in $\KK$ of an element of 
$\AA(\R^n)$, which is the subject of a general result on algebras w.m.v.\ proved herein (see Lemma~\ref{L:zeromeasure} below). To illustrate this problem, we briefly exhibit here  a  very simple example in the periodic context.    So, let us consider the homogenization of the strongly degenerate equation 
$$
u_t=\Delta (F(u)+\psi_0(\frac{x}{\ve})),
$$
where 
$$
F(u)=\begin{cases} u+\frac12,& u<-\frac12\\ 0,&-\frac12\le u\le \frac12,\\ u-\frac12,& u>\frac12, \end{cases}
$$
and $\psi_0:\R\to\R$ is the periodic function of period 4 defined for $x\in[-2,2]$ by
$$
\psi_0(x)=\begin{cases} -x-2,& -2\le x\le -1,\\ x, & -1\le x\le 1,\\ -x+2, &1\le x\le 2. \end{cases}
$$
Such nonlinear flux function is a prototype for models of the so called Stefan problem (see, e.g., \cite{Dam}).  Our homogenization analysis of such strongly degenerate equations implies in this simple case that the homogenized equation is  
$$
u_t=\Delta \bar f(u),
$$
where
$$
\bar f(u)=\begin{cases} u+\frac12,& u<-\frac{3}2,\\ \frac23 u,&-\frac{3}2\le u\le \frac 32,\\ u-\frac12, &u>\frac32.\end{cases}
$$
So, although the equations to be homogenized are strongly degenerate, the homogenized equation is nondegenerate.  The reason for this  is basically the fact that the level sets of the periodic function $\psi_0$ defined above have Lebesgue measure zero. As remarked after the proof of Lemma~\ref{L:zeromeasure} below, 
if $\AA(\R^n)$ is an algebra w.m.v.\ containing the periodic function $\psi_0$, then small perturbations of the form $\psi=\psi_0+\d \psi_1$, with $\psi_1\in\AA(\R^n)$, will satisfy the zero measure condition on the level sets in $\KK$, which yields a similar nice behavior of the homogenized equation.

Now we make some comments on the techniques used in the proof of Theorems~\ref{T:6.1} and~\ref{T:7.1}.  

For the proof of Theorem~\ref{T:6.1}, the technique used goes back to the work of E and Serre~\cite{ES}, on the periodic homogenization of the one-dimensional conservation law $u_t+(f(u)-V(\frac{x}{\ve}))_x=0$ (see \cite{AF}, for a multidimensional extension to almost periodic homogenization), which in turn is inspired in the work of DiPerna~\cite{DP} on the uniqueness of measure-valued solutions of scalar conservation laws. 
Although it requires the restriction to well-prepared initial data, the technique is otherwise very powerful, since it applies to any ergodic algebra, to unbounded domains, and it implies directly the existence of strong correctors.  This method strongly relies on the concept of two-scale Young measures  first introduced, in the periodic context, by W.~E in \cite{E}, as a nonlinear extension of the concept of two-scale convergence introduced by Nguetseng in \cite{N} and further developed by Allaire in \cite{A}.  The extensions of two-scale Young measures  to almost periodic functions and to general algebras with mean value were established in \cite{AF} and \cite{AFS}, respectively.  

Concerning the proof of Theorem~\ref{T:7.1}, the method applied in this case is completely different from the one for the proof of Theorem~\ref{T:7.1} and it is based on the conversion of the homogenization of the nonlinear parabolic equation into the homogenization of a corresponding fully nonlinear parabolic equation, of a particular simple type, by applying the inverse Laplacian operator on the equation.  To homogenize the corresponding particular fully nonlinear parabolic equation, we use ideas that go back to Evans~\cite{EvH} and Ishii~\cite{I}, among others. Here, the initial data  do not need to be ``well-prepared''.  On the other hand, because of the extensive use of the inverse Laplacian operator on the whole domain, we have to restrict the analysis to  bounded domains. We also have to restrict the homogenization analysis to regular algebras with mean value. The latter is a concept introduced here, whose largest representative so far known is  the Fourier-Stieltjes algebra studied in \cite{FS}. As pointed out in \cite{FS}, the Fourier-Stieltjes algebra {\em strictly} contains the algebra of perturbed almost periodic functions, whose elements can be written as the sum of an almost periodic function and a continuous function vanishing at infinity. Here, we give the easy proof of the fact  that  a regular algebra w.m.v.\  is ergodic.  We further remark that, to obtain the corrector property \eqref{eI.07}, also in this case, essential use is made of the two-scale Young measures, combined with a clever argument by Visintin in \cite{Vi}.

Before concluding this introduction,  we would like to mention that in this paper we also make a detailed review and provide some new results on the existence and stability theory for degenerate parabolic equations of the type considered here. We do that   because  we need some specific results that are not proved elsewhere,  also just to introduce some notations used later on, as well as in order to have our work the most self-contained possible.   

This paper is organized as follows. In Section~\ref{S:1} we recall the concepts of algebra w.m.v., generalized Besicovitch space and ergodic algebra. We also recall  a general result established in \cite{AFS} which relates such algebras and the translation operators acting on them with the continuous functions defined on certain compact spaces and certain groups of homeomorphisms of these compact spaces. In Section~\ref{S:2}, we introduce the concept of regular algebra w.m.v., prove that these are ergodic algebras, and that this concept includes the Fourier-Stieltjes spaces $\FS(\R^n)$. In Section~\ref{S:3}, we briefly recall the general result of \cite{AFS} on the existence of two-scale Young measures associated with a given algebra w.m.v. In Section~\ref{S:4}, we provided a self-contained discussion about the well-posedness of the the initial-boundary value problem with null pressure boundary condition for degenerate porous medium type equations. In Section~\ref{S:5}, we prove Theorem~\ref{T:6.1} on the homogenization of \eqref{eI.01} on an unbounded domain $\Om$, for a general ergodic algebra, under the restriction to well-prepared initial data. Finally, in Section~\ref{S:6}, we prove Theorem~\ref{T:7.1} on the homogenization of \eqref{eI.01} on a bounded domain $\Om$, for regular algebras w.m.v., but for general initial data. 

\section{Ergodic Algebras}\label{S:1}

In this section we recall some  basic facts about algebras with mean values and ergodic algebras that will be needed for the purposes of this paper. To begin with, we recall 
the notion of mean value for functions defined in $\re^n$. 

\begin{definition}\label{D:3} Let $g\in L_\loc^1(\R^n)$. A number $M(g)$ is called the {\em mean value of $g$} if
\begin{equation}\label{e1.2}
\lim_{\ve \to0} \int_Ag(\ve^{-1}x)\,dx=|A|M(g)
\end{equation}
for any Lebesgue measurable bounded set $A\subset\R^n$, where $|A|$ stands for the Lebesgue measure of $A$.
This is the same as saying that $g(\ve^{-1}x)$ converges, in the
duality with $L^\infty$ and compactly supported functions, to the constant $M(g)$.
Also,
if $A_t:=\{x\in\R^n\,:\, t^{-1}x\in A\}$ for $t>0$ and $|A|\ne0$, \eqref{e1.2} may be written as
\begin{equation}\label{e1.3}
\lim_{t\to\infty}\frac1{t^n|A|}\int_{A_t}g(x)\,dx=M(g).
\end{equation}
\end{definition}
We will also use the notation $\medint_{\R^n}g\,dx$ for $M(g)$.

\medskip
 As usual, we denote by $\BUC(\R^n)$ the
space of the bounded uniformly continuous real-valued functions in
$\R^n$.

We recall now the definition of algebra with mean value introduced in \cite{ZK}. 

\begin{definition}\label{D:5} Let $\AA$ be a linear subspace of $\BUC(\R^n)$.
We say that $\AA$ is an {\em algebra with mean value} (or {\em
algebra w.m.v.}, in short), if the following conditions are
satisfied:
\begin{enumerate}
\item[(A)] If $f$ and $g$ belong to $\AA$, then the product $fg$ belongs to $\AA$.
\item[(B)] $\AA$ is invariant under the translations $\tau_y:\R^n\to\R^n$, $x\mapsto x+y$, $y\in\R^n$, that is, if $f\in\AA$, then $\tau_y f\in\AA$, for all $y\in\R^n$, where $\tau_yf:=f\circ \tau_y$, $f\in\AA$.
\item[(C)] Any $f\in\AA$ possesses a mean value.
\item[(D)] $\AA$ is closed in $\BUC(\R^n)$ and contains the unity, i.e., the function $e(x):=1$ for $x\in\R^n$.
\end{enumerate}
\end{definition}

\begin{remark}\label{R:awmv} Condition (D) was not originally in \cite{ZK} but  its inclusion does not change matters since any algebra satisfying (A), (B) and (C) can be extended  to
an algebra satisfying (A)--(D) in an unique way modulo isomorphisms.
\end{remark}

For the development of the homogenization theory in algebras 
with mean value, as it is done in \cite{ZK,JKO} (see also \cite{CG,AFS}),
in similarity with the case of almost periodic functions, one
introduces, for $1\leq p<\infty$, the space  $\BB^p$ as the abstract
completion of the algebra 
$\AA$ with respect to the Besicovitch seminorm
$$
|f|_p:= \left(\Medint_{\R^n}|f|^p\,dx\right)^{1/p}
$$
Both the action of translations and the mean value
extend by continuity to $\BB^p$, and we will keep using the notation 
$\tau_y f$ and $M(f)$ even when $f\in\BB^p$. Furthermore,
for $p>1$ the product in the algebra extends to a bilinear operator from $\BB^p\times\BB^q$ into $\BB^1$,
with $q$ equal to the dual exponent of $p$, satisfying
$$
|fg|_1\leq |f|_p|g|_q.
$$
In particular, the operator $M(fg)$ provides a nonnegative definite
bilinear form on $\BB^2$.

Since there is an obvious inclusion between elements of this family of spaces,
we may define the space $\BB^\infty$ as follows:
$$
\BB^\infty=\{f\in \bigcap_{1\leq p<\infty}\BB^p\,:\,\sup_{1\le p<\infty}|f|_p<\infty\},
$$
We endow $\BB^\infty$  with the (semi)norm
$$
|f|_\infty:=\sup_{1\le p<\infty}|f|_p.
$$
Obviously the corresponding quotient spaces for all these spaces
(with respect to the null space of the seminorms) are Banach spaces,
and in the case $p=2$ we obtain a Hilbert space. We denote by
$\overset{\BB^p}{=}$, the equivalence relation given by the equality
in the sense of the $\BB^p$ semi-norm. We will keep the notation $\BB^p$ also
for the corresponding quotient spaces.

\begin{remark}\label{R:0.1} A classical argument going back to Besicovitch~\cite{B} (see also \cite{JKO}, p.239) shows that the elements of $\BB^p$ can be represented
by functions in $L_{\loc}^p(\R^n)$, $1\le p<\infty$. 
\end{remark}

We next recall a result established in \cite{AFS} which provides a connection between algebras with mean value and the algebra $C(\CK)$ of continuous functions on a certain compact (Hausdorff) topological space. We state here only the parts of the corresponding result in \cite{AFS} that will be used in this paper. 

\begin{theorem}[cf.\ \cite{AFS}]\label{T1} 
For an algebra $\AA$, we have:
\begin{enumerate}
\item[(i)] There exist a compact space
${\mathcal K}$ and an isometric isomorphism $i$ identifying $\AA$ with the
algebra $C({\mathcal K})$ of continuous functions on ${\mathcal K}$.
\item[(ii)] The translations $\tau_y:\R^n\to\R^n$, $\tau_y x=x+y$,
induce a family of homeomorphisms $T(y):{\mathcal K}\to{\mathcal K}$, $y\in\R^n$, satisfying the group properties $T(0)=I$, $T(x+y)=T(x)\circ T(y)$, such that the mapping $T:\R^n\X\CK\to\CK$, $T(y,z):=T(y)z$, is continuous.
\item[(iii)] The mean value on $\AA$ extends to  a Radon probability measure ${\mathfrak m}$ on 
${\mathcal K}$ defined by
$$
\int_{\mathcal K}i(f)\,d\mathfrak m :=\Medint_{\R^n}f\,dx, \qquad f\in\AA,
$$
which is invariant by the
group of homeomorphisms $T(y):\CK\to\CK$, $y\in\R^n$, that is, $\mm(T(y) E)=\mm(E)$ for all Borel sets $E\subset\CK$.
\item[(iv)]For $1\le p\le \infty$, the
Besicovitch space $\BB^p\big/\overset{\BB^p}{=}$ is
isometrically isomorphic to $L^p({\mathcal K}, {\mathfrak m})$.
\end{enumerate}
\end{theorem}

A function $f\in\BB^2$ is  said to be {\em invariant} if
$\tau_y f\overset{\BB^2}{=} f$, for all $y\in\R^n$. More clearly,
$f\in\BB^2$ is invariant if
\begin{equation}\label{e1.INV}
M\bigl(|\tau_yf-f|^2\bigr)=0,\qquad \text{for all $y\in\R^n$}.
\end{equation}
The concept of ergodic algebra is then introduced as follows.

\begin{definition}\label{D:6} An  algebra w.m.v.\ $\AA$  is called an {\em ergodic algebra} if any invariant function
$f$ belonging to the corresponding space $\BB^2$ is equivalent (in $\BB^2$) to a constant.
\end{definition}

A very useful alternative definition of ergodic algebra is also given in \cite{JKO}, p.~247,
and shown therein to be equivalent to
Definition~\ref{D:6}. 
We state that as the following lemma.

\begin{lemma}[cf.\ \cite{JKO}]\label{L:ergo} Let $\AA\subset\BUC(\R^n)$
be an algebra w.m.v.. Then $\AA$ is ergodic
if and only if
\begin{equation}\label{e.Lergo}
\lim_{t\to\infty}M_y\left(\bigl|\frac{1}{|B(0;t)|}
\int_{B(0;t)}f(x+y)\,dx-M(f)\bigr|^2\right)=0
\qquad\forall f\in\AA,
\end{equation}
where $M_y$ denotes the mean with respect to the variable $y$.
\end{lemma}

\begin{remark} \label{R:S1} As examples of ergodic algebras, besides the trivial one of the periodic functions,  the already mentioned example of the almost periodic functions, $\AP(\R^n)$, and the larger space of the Fourier-Stieltjes transforms, $\FS(\R^n)$, which will be commented in Section~\ref{S:2}, an even larger ergodic algebra, including all the just mentioned ones, is the space of the weakly almost periodic functions $\WAP(\R^n)$, introduced by Eberlein, in \cite{Eb}. This space is defined as the subspace of the functions in $\varphi\in C(\R^n)$, whose family of translates 
$\varphi(\cdot+\l)$, $\l\in\R^n$, is weakly pre-compact in $C(\R^n)$. Since weak convergence in $C(\R^n)$ is equivalent to pointwise convergence in $C(\hat\KK)$, where $\hat\KK$ is the  \v Cech compactification of $\R^n$, generated by the algebra $C(\R^n)$,  we easily see that $\WAP(\R^n)$ is, indeed, an algebra, invariant by translations. Also, the fact that 
$\AP(\R^n)\subset\WAP(\R^n)$, follows from  a well known result of Bochner (see, e.g., \cite{Bo}), establishing the property of strong pre-compactness in $C(\R^n)$ of the translates, as equivalent to the definition of almost periodic functions.  The theory developed in \cite{Eb} shows, in fact,  that $\WAP(\R^n)$ is an ergodic algebra including $\FS(\R^n)$. The latter inclusion is strict according to a result of Rudin in \cite{Ru}.  
\end{remark}

The following lemma from \cite{FS1}  will be used subsequently, in our discussion about the measure of level sets of the elements of an algebra w.m.v. We give an outline of its proof for the reader's convenience. 

\begin{lemma}[cf.\ \cite{FS1}]\label{L:FS1} Let $\AA(\R^n)$ be an algebra w.m.v.\ in $\R^n$ and $\xi:\R^n\to\R^n$ be a vector field such that $\xi_i\in \AA(\R^n)\cap \Lip(\R^n)$,  for
 $i=1,\cdots,n$. 
Let $\Phi:\R^n\X\R\to\R^n$ be the flow generated by $\xi$, that is,  for any $(x_0,t_0)\in\R^n\X\R$, $\Phi(x_0,t_0)=x(t_0;x_0)$, where $x(t;x_0)$ is the solution of 
\begin{equation}\label{eLFS1}
\begin{cases}
\dfrac{dx}{dt}=\xi(x),\\
x(0)=x_0.
\end{cases}
\end{equation}
Then, for any $g\in\AA(\R^n)$ and $t\in\R$, $g\circ\Phi_t\in\AA(\R^n)$, with $\Phi_t(x)=\Phi(x,t)$, $\Phi_t$ extends to a homeomorphism $\Phi_t:\KK\to\KK$, and $\Phi$ extends to a continuous mapping  $\Phi:\KK\X\R\to\KK$.
\end{lemma}
  
\begin{proof} The assertion is equivalent to saying that, for any $\varphi\in\AA(\R^n)$, $\varphi\circ \Phi:\R^n\X[-T,T]\to\R$ extends, in a unique way, to a function $\varphi\circ\Phi\in C(\KK\X[-T,T])$, for any $T>0$.  First, we claim that, given any $\z\in\AA(\R^n;\R^n)$, we have  that $\varphi(x+\z(x))\in \AA(\R^n)$, or, equivalently, that $\varphi(\cdot+\z(\cdot))\in C(\KK)$, if we view $\z$ and $\varphi$ as extended to functions in $C(\KK;\R^n)$ and $C(\KK)$, respectively. Indeed, the claim is a direct application of Theorem~\ref{T1}, since, viewed as a function on $\KK$, $\varphi(\cdot+\z(\cdot))=\varphi\circ T(\z(\cdot),\cdot)$, where $T:\R^n\X\KK\to\R^n$ is the continuous mapping extending the translations $\R^n\X\R^n\to\R^n$, $(y,x)\mapsto x+y$. 

Now,  by the group property $\Phi(x,t+s)=\Phi(\Phi(x,t),s)$, it suffices to prove that $\Phi$ extends to a continuous mapping $\KK\X[-T,T]\to\KK$, for $T>0$ as small as we wish. 
Thus, we begin by recalling that $\Phi$ satisfies
\begin{equation}\label{eFS1.1}
\Phi(x,t)=x+\int_0^t \xi(\Phi(x,s))\,ds,
\end{equation}
and for, $t\in[-T,T]$, with $T>0$ sufficiently small, we may obtain $\Phi$ using Banach fixed point theorem, as the limit  of a sequence of continuous mappings 
$\Phi^j:\R^n\X[-T,T]\to\R^n$, in the metric space 
$$
\E:=\{\Psi:\R^n\X[-T,T]\to\R^n\,:\, \Psi(x,t)-x\in \BUC(\R^n\X[-T,T];\R^n)\},
$$
 endowed with the metric 
\begin{equation}\label{eFS1.2}
d(\Phi_1,\Phi_2)=\sup_{(x,t)\in\R^n\X[-T,T]}|\Phi_1(x,t)-\Phi_2(x,t)|,
\end{equation}
where $\Phi^j$ is defined recursively by
\begin{equation}\label{eFS1.3}
\Phi^{j+1}(x,t):= x+\int_0^t\xi(\Phi^j(x,s))\,ds,\quad \Phi^0(x,t)=x,\ (x,t)\in\R^n\X[-T,T].
\end{equation}
Therefore, all that remains is to prove by induction that each $\Phi^j$ extends to a continuous mapping $\KK\X[-T,T]\to\KK$, for all $j\in\N$. Indeed, for such mappings, convergence in  the metric \eqref{eFS1.2} clearly implies convergence in the topology generated by the family of pseudo-metrics 
\begin{equation}\label{eFS1.4}
d_\varphi(\Phi_1,\Phi_2)=\sup_{(z,t)\in\KK\X[-T,T]}|\varphi(\Phi_1(z,t))-\varphi(\Phi_2(z,t))|,\quad \varphi\in C(\KK),
\end{equation}
and, so, $\Phi$, as the limit in the topology given by \eqref{eFS1.4}, of the sequence of continuous mappings $\Phi^j:\KK\X[-T,T]\to\KK$, will also be a continuous mapping
 $\KK\X[-T,T]\to\KK$. 
 
  Now, we have already proved that $\Phi^1(x,t)=x+t\xi(x)$ extends to a continuous mapping $\KK\X[-T,T]\to\KK$, since, for any $\varphi\in\AA(\R^n)$, $\varphi(\cdot+t\xi(\cdot))\in\AA(\R^n)$, for each fixed $t\in[-T,T]$, and the uniform continuity of $\varphi$ immediately implies that $\varphi\circ\Phi^1\in C(\KK\X[-T,T])$, where as usual, we use freely the identification $\AA(\R^n)\sim C(\KK)$.
  
  Finally, we have to check that the induction hypothesis that $\Phi^j$ extends to a continuous mapping $\KK\X[-T,T]\to\KK$, implies that $\Phi^{j+1}$ also extends to a continuos mapping $\KK\X[-T,T]\to\KK$. Indeed, we first prove that, if $\Phi^j$   extends to a continuous mapping $\KK\X[-T,T]\to\KK$, then the function
\begin{equation}\label{eFS1.5}
\z(x,t):= \int_0^t\xi(\Phi^j(x,s))\,ds,
\end{equation}
satisfies, $\z(\cdot,t) \in\AA(\R^n;\R^n)$ for each $t\in[-T,T]$, and  $\z\in C([-T,T]; \AA(\R^n;\R^n))$. In fact, the integral from 0 to $t$ defining $\z(\cdot,t)$, may be defined as the limit of Riemann sums, each of which is clearly a function in $\AA(\R^n;\R^n)$, and these Riemman sums converge uniformly in $\R^n$, by the assumption that $\Phi^j$ extends as a continuous mapping $\KK\X[-T,T]\to\KK$, which implies that $\xi\circ\Phi^j\in C([-T,T];\AA(\R^n;\R^n))$. Hence, we have, as asserted, that $\z\in C([-T,T]; \AA(\R^n;\R^n))$. 
In conclusion, since $\z\in C([-T,T];\AA(\R^n;\R^n))$, given any $\varphi\in\AA(\R^n)$, we have $\psi(\cdot,t):=\varphi(\cdot+\z(\cdot,t))\in\AA(\R^n)$, by what has already been proved,
and, by the uniform continuity of $\varphi$, $\psi\in C([-T,T];\AA(\R^n))$, which is the same to say that $\Phi^{j+1}$ extends to a continuous mapping $\KK\X[-T,T]\to\KK$, finishing the proof.

\end{proof}

We close this section  establishing a general result concerning algebras w.m.v.\ which will be used in our investigation on the homogenization of porous medium type equations in the last two sections of the present work.  We first establish the following definition.

\begin{definition}\label{D:sreg}  For  a $C^1$  function $\psi$, belonging to algebra w.m.v.\  $\AA(\R^n)$, such that $\partial_{x_i} \psi \in\AA(\R^n)$, $i=1,\dots,n$,  we say that 
$\a\in\R$ is a {\em strongly regular} value of  $\psi$ if there exists $\d_\a>0$ such  
$|\psi(x)-\a|^2+|\nabla \psi(x)|^2>\d_\a$, for all $x\in\R^n$, where $|\nabla\psi(x)|^2=\sum_{i=1}^n(\partial_{x_i}\psi(x))^2$.
\end{definition}

\begin{lemma}\label{L:zeromeasure}  Let $\AA(\R^n)$ be an algebra w.m.v.\ and $\psi\in\AA(\R^n)$ be such that $\partial_{x_i}\psi, \partial_{x_ix_j}^2\psi  \in
\AA(\R^n)$, $i,j=1,\dots,n$.  
\begin{enumerate}
\item[(i)]  If  $\a\in\R$ is a strongly regular value of $\psi$, then
\begin{equation}\label{ezeromeas}
\mm\left(\{z\in \KK:\,  \psi(z)=\a \}\right)=0,
\end{equation}
where $\KK$ is the compact space given by Theorem~\ref{T1} and $\mm$ is the associated invariant probability measure on $\KK$.  
\item[(ii)]  If
\begin{equation}\label{ezeromeas2}
\mm\left(\{z\in \KK:\, |\nabla \psi(z)|=0 \}\right)=0,
\end{equation}
then \eqref{ezeromeas} holds for all $\a\in\R$. In particular, this is the case if $0$ is a strongly regular value of $\po_{x_k} \psi$, for some $k\in\{1,\cdots,n\}$.
 \end{enumerate}

\end{lemma}

\begin{proof}  We first prove item (i).  By the hypotheses and the properties of the algebras w.m.v., we have that $|\nabla \psi|\in\AA(\R^n)$ and so it extends to a function in 
$C(\KK)$. Since $\a$ is a strongly regular value of $\psi$, the sets
$A=\{z\in\KK\,:\, \psi(z)=\a\}$ and $B=\{ z\in\KK\,:\, |\nabla\psi(z)|=0\}$ are two disjoint compact subsets of $\KK$. Hence, there exists $\d_0>0$ such that 
\begin{equation}\label{eL2.2.1}
A\subset \V:=\{z\in\KK\,:\, |\nabla\psi(z)|>\d_0\}. 
\end{equation}
In particular, given any  $z\in A$, there exists $j\in\{1,\dots,n\}$ such that $|\partial_{x_j}\psi(z)|>\d_0/n$.  

We claim that
$$
A\cap\R^n\subset \bigcup_{j=1}^n\bigcup_{k=1}^\infty S_k^j, 
$$ 
where each $S_k^j$ is the graph of a  $C^1$ function defined on an open subset of the  space of the $n-1$ variables 
$$
(x_1,\dots,x_{j-1},x_{j+1},\dots,x_n).
$$
Moreover,  for each  fixed  $j\in\{1,\dots, n\}$, the sets  $S_k^j$, $k\in\Z$,  are separated from each other along the lines parallel to the $x_j$-axis by  a distance greater than a positive number $2\s_0$, in the sense that,  if  $x_1\in S_{k_1}^j$, $x_2\in S_{k_2}^j$, $k_1\ne k_2$, and $x_1-x_2= se_j$, where $e_j$ is the $j$-th element of the canonical basis,
then $|s|>2\s_0$.  

Indeed, let 
\begin{equation}\label{eAj}
A^j=\{x\in A\cap\R^n\,:\,|\partial_{x_j}\psi(x)|>\frac{\d_0}{n}\}.
\end{equation}
 Clearly $A\cap\R^n=\cup_{j=1}^n A^j$.  By the Implicit  Function Theorem, we have that $A^j$ is the union of a family of connected graphs of $C^1$ functions.  
Moreover, since $\partial_{x_j}\psi$ is uniformly continuous, there exists $\s_0$  such that $|x-y|<\s_0$ implies 
$|\partial_{x_j}\psi(x)-\partial_{x_j}\psi(y)|<\d_0/2n$. Therefore, by \eqref{eAj}, any two points $x,y\in A_j$ lying both in  one  line parallel to the $x_j$-axis must satisfy $|x-y|>2\s_0$.
 In particular, the set of connected graphs in $A^j$ is countable, since for each point in the hyperplane $\{x_j=0\}$ with rational coordinates there corresponds at most a countable number of graphs whose projection in $\{x_j=0\}$ contains that point.
Now, given
a connected graph contained in $A^j$, by Zorn's lemma, we can obtain a maximal family of connected graphs in $A_j$, containing the given graph, whose projections into the hyperplane $\{x_j=0\}$ are disjoint from each other. We call this maximal family $S_1^j$.  We then consider the family of connected graphs $A^j\setminus S_1^j$ and pick up 
a connected graph from it. Again by Zorn's lemma such graph belongs to a maximal family of connected graphs in $A^j\setminus S_1^j$  whose projections into the hyperplane 
$\{x_j=0\}$ are pairwise disjoint. We call this maximal family $S_2^j$. We then consider the family of connected graphs $A^j\setminus (S_1^j\cup S_2^j)$ and, from it, we define a maximal family of connected graphs whose projections in the hyperplane $\{x_j=0\}$ are pairwise disjoint, call this maximal family $S_3^j$, and so on. In this way, relabeling if necessary, we end up decomposing $A^j$ into a disjoint union, $\cup_{k=1}^\infty S_k^j$, of maximal families of connected graphs whose projections in the hyperplane $\{x_j=0\}$ do not intersect each other. Clearly, each such maximal family, $S_k^j$,  maybe be viewed as a graph of a $C^1$ function defined on an open subset  of the space of the variables $x_1,\dots,x_{j-1},x_{j+1},\dots, x_n$.  By the proof it is clear also the assertion concerning the separation of the $S_k^j$, along lines parallel to $e_j$.

Let $\Phi:\R^n\X\R\to\R^n$ be the flow generated by $\nabla \psi$, which, by Lemma~\ref{L:FS1}, may be extended to a continuous mapping $\Phi:\KK\X\R\to\K$. Since $A$ is compact in $\KK$, $\V$ is open in $\KK$, and $A\subset\V$, by continuity, for $\tau_0>0$ sufficiently small, we have $\V(\tau):=\Phi(A\X[-\tau,\tau])\subset\V$, for all $0<\tau\le \tau_0$. 
We decompose $\V_\tau$ as 
$$
\V(\tau)=\bigcup_{j=1}^n\bigcup_{k=1}^\infty V_k^j(\tau),
$$
where
$$
V_k^j(\tau):=\{ x\in\R^n\, :\, x=\Phi(x',s),\ x'\in S_k^j,\ s\in [-\tau,\tau]\}.
$$
Let $\ve_0>0$ be such that the compact sets $A_\ve=\{z\in\KK\,:\,|\psi(z)-\a|\le \ve\}$ satisfy $A_\ve\subset\V$, for $0<\ve<\ve_0$. Such $\ve_0>0$ exists, by compactness arguments,  since 
$$
\bigcap_{0<\ve<1}A_\ve=A.
$$
For $0<\ve<\ve_0$, such that $2\ve/\d_0^2<\tau_0$,  let us define
$$
\phi_\ve(x)=\min\{|\psi(x)-\a|,\ \ve\},\qquad \psi_\ve(x)=1-\ve^{-1}\phi_\ve(x).
$$ 
The non-negative  function  $\psi_\ve$, so defined, is clearly an element of $\AA(\R^n)$, which  is equal 1 on $A$, and whose support is $A_\ve$. We also have that $A_\ve\subset\V(\tau_0)$. Indeed, given any $x_1\in A_\ve$, assume for concreteness that $\psi(x_1)<\a$, and let us consider the curve $\Phi(x_1,t)$, for $t\ge 0$, and the function 
$$
\g(t):=\psi(\Phi(x_1,t)), \ t\ge0.
$$
We have $\g'(t)>\d_0^2$, while $\Phi(x_1,t)\in \V$. So, either $\g(t_0)=\a$, for some $t_0>0$, or $\Phi(x_1,t)$ leaves $\V$ before $\g$ achieves the value $\a$, which is impossible
since $\g$ is increasing for $0<t<t_*$, where $t_*$ is the least time in which $\Phi(x_1,t)$ leaves $\V$, and $A_\ve\subset \V$, so $\Phi(x_1,t)$ could not have left $\V$ without passing through $A=\{\psi=\a\}$.   Also,  if $x_0\in A$ and $x_1=\Phi(x_0,t_0)$, then 
$$
|\psi(x_1)-\a|=|\int_0^{t_0}|\nabla\psi(\Phi(x_0,t))|^2\,dt|\ge \d_0^2 |t_0|.
$$  
Therefore,
$$
A_\ve=\supp\psi_\ve\subset  \V(\frac{2\ve}{\d_0^2})=\bigcup_{j=1}^n\bigcup_{k=1}^\infty V_k^j(\frac{2\ve}{\d_0^2}).
$$

Let us denote $V_k^{j,\ve}:=V_k^j(\frac{2\ve}{\d_0^2})$.  We have 
\begin{equation}\label{eL22.1}
\mm(A)\le \mm(\{z\in\KK\,:\,|\psi_\ve(z)|>\frac12\})\le 2\int_{\KK}\psi_\ve(z)\,d\mm(z)=2\lim_{R\to\infty}\frac{1}{|B(0;R)|}\int_{|x|<R}\psi_\ve(x)\,dx.
\end{equation}
Now, for each $j\in\{1,\dots,n\}$, $\# \{V_k^{j,\ve}\,:\, V_k^{j,\ve}\cap B(0;R)\ne \emptyset\}< \frac{R}{\s_0}$, and clearly
$$
\frac{1}{R^{n-1}}\H^{n-1}\left( B(0;R)\cap S_k^j\right)\le C,
$$
for some $C>0$ depending only on $\psi$. Hence, for any $R>0$, we have
\begin{align*}
\frac{1}{|B(0,R)|}\int_{|x|<R}\psi_\ve(x)\,dx&\le\sum_{j=1}^n\sum_{k\in\Z}\frac{1}{|B(0;R)|}\int_{B(0;R)\cap V_k^{j,\ve}}\psi_\ve(x)\,dx\\
                                                                          &\le \sum_{j=1}^n\frac{R}{\s_0|B(0;R)|}\max_k\int_{B(0;R)\cap V_k^{j,\ve}}\psi_\ve(x)\,dx\\
                                                                          &\le\sum_{j=1}^n \frac{R}{\s_0|B(0;R)|}\max_k\int_{B(0;R)\cap V_k^{j,\ve}}\,dx \\                                                                          
                                                                          &< C\ve,
\end{align*}
again for some $C>0$  depending only on $\psi$. Thus, we get  that $\mm(A) < C\ve$, and, since $\ve>0$ is arbitrary, we arrive at the desired conclusion, ending the proof of (i).  

As for the proof of (ii), by assumption \eqref{ezeromeas2}, we only need to prove that 
$$
\mm\left(\{z\in \KK:\,  \psi(z)=\a,\ |\nabla\psi(z)|>0 \}\right)=0.
$$
But, we may write 
$$
\{z\in \KK:\,  \psi(z)=\a,\ |\nabla\psi(z)|>0 \}=\bigcup_{l=1}^\infty B^l, \qquad  B^l=\{z\in\KK\,:\, \psi(z)=\a,\  |\nabla \psi(z)|\ge \frac1l\}.
$$
Now, we claim that $\mm(B^l)=0$, $l=1,2,\dots$. Indeed, the claim follows by arguments totally similar to those used in the proof of (i).  The only nontrivial adaptation to be made, is that,
 instead of using the function $\psi_\ve$ defined above, we shall now use the function
 $$
 \tilde \psi_\ve(x)=\psi_\ve(x) \theta_l(x),\qquad \theta_l(x):=2l\left( \max\left\{\frac1{2l},\ \min\{\ |\nabla\psi(x)|,\ \frac1{l}\} \right\}-\frac1{2l}\right).
 $$ 
 We then get an inequality similar to \eqref{eL22.1} with $A$ replaced by $B^l$ and $\psi_\ve$ replaced by $\tilde\psi_\ve$.  We also define the analogues of $S_k^j$ and $V_k^j(\tau)$  and the remaining of the proof follows as in the proof of (i). It is also clear that, if $0$ is a strongly regular value of $\po_{x_k}\psi$, for some $k\in\{1,\dots,n\}$, then 
 \eqref{ezeromeas2} holds, by (i).   This finishes the proof.
 
\end{proof}

\begin{remark} \label{R:zeromeasure} In general, for any $\psi$ in an algebra w.m.v.\  $\AA(\R^n)$, we trivially have $\mm(\{z\in\KK\,:\,\psi(z)=\a\})=0$, except for a countable set of $\a$'s. Nevertheless, in general we do not have any other information  about the set of exceptional $\a$'s besides the fact that it is countable; in particular, it could be dense in $\R$.
However, we can use Lemma~\ref{L:zeromeasure} to provide examples where the set of exceptional $\a$'s is empty. For instance, if $\psi_0$ is a $C^2$ periodic function in $\R^n$ for which 0 is a regular value of $\nabla\psi_0$ in the usual sense, then, by the Implicit Function Theorem, we know that  the set $\{x\in\R^n\,:\,|\nabla\psi_0(x)|=0\}$ has $n$-dimensional Lebesgue measure zero, and since $\nabla\psi_0=0$ almost everywhere on the level sets $\{\psi_0=\a\}$, we conclude that all these level sets have $n$-dimensional Lebesgue measure zero.  Now, if $\AA(\R^n)$ is an algebra w.m.v.\ containing  such a  periodic function $\psi_0$ and $\psi_1,\partial_{x_i}\psi_1,\partial^2_{x_i x_j}\psi_1\in\AA(\R^n)$, $i,j=1,\dots,n$, then, for $\d>0$ sufficiently small, we have that $\psi=\psi_0+\d\psi_1$ satisfies the hypotheses of the item (ii) of Lemma~\ref{L:zeromeasure}, and so the conclusion of (ii) holds for $\psi$.  

\end{remark}

\section{Regular algebras w.m.v.\ and the Fourier-Stieltjes space $\FSM(\R^n)$.}\label{S:2}

In this section we introduce the concept of regular algebra w.m.v.\ and recall the definition and some basic  properties of the Fourier-Stieltjes space  introduced by the authors in~\cite{FS}, which is, to the best of our knowledge, the largest known example of a  regular algebra w.m.v..

For any $f\in L^\infty(\R^n)$, let
us denote by $\hat f$ the Fourier transform of $f$ defined as the following distribution
$$
\la \hat f,\phi\ra:=\int f(x)\hat\phi(x)\,dx,\qquad \text{for all $\phi\in C_c^\infty(\R^n)$},
$$
where $\hat \phi$ denotes the usual  Fourier transform of $\phi$, i.e.,
$$
\hat\phi(x)=\frac{1}{(2\pi)^{\frac{n}{2}}}\int \phi(y)e^{-iy\cdot x}\,dx.
$$

Given an algebra w.m.v.\ $\AA$, let us denote by $V(\AA)$ the subspace formed by the elements $f\in\AA$ such that $M(f)=0$, namely,
$$
V(\AA):=\{ f\in\AA\,:\, M(f)=0\}.
$$
Also, let us denote by $Z(\AA)$ the subset of those $f\in\AA$ such that the distribution $\hat f$ has compact support not containing the origin $0$, that is,
$$
Z(\AA):=\{f\in\AA\,:\, \supp(\hat f)\ \text{is compact and}\ 0\notin\supp(\hat f)\}.
$$

We collect in the following lemma some useful properties of the functions in $Z(\AA)$, whose proof is found in \cite{JKO}, p.~246.

\begin{lemma}[cf.\ \cite{JKO}]\label{L:za} Let $\AA$ be an algebra w.m.v.\ in $\R^n$ and $f\in Z(\AA)$. Then:  
\begin{enumerate} 
\item[(i)] There exists $u\in C^\infty(\R^n)\cap Z(\AA)$ such that $\Delta u=f$, where $\Delta$ is the usual Laplace operator in $\R^n$;     $u=f*\zeta$
for certain smooth function $\zeta$, fast decaying together with all its derivatives, satisfying $\hat \zeta\in C_c^\infty(\R^n)$ and $0\notin \supp(\hat \zeta)$.
\item[(ii)] For any Borelian $Q\subset\R^n$, with $|Q|>0$, we have
\begin{equation}\label{e.uni}
\lim_{t\to\infty}\frac{1}{t^n|Q|}\int_{Q_t}f(x+y)\,dx=0, \qquad\text{uniformly in $y\in\R^n$}.
\end{equation}
In particular, $Z(\AA)\subset V(\AA)$. 
\end{enumerate}
\end{lemma}

The fundamental result about ergodic algebras, proved by Zhikov and Krivenko~\cite{ZK}, is the following.

\begin{theorem}[cf.\ \cite{ZK}]\label{T:ZK} If $\AA$ is an ergodic algebra, then $Z(\AA)$ is dense in $V(\AA)$ in the topology of the corresponding space $\BB^2$. 
\end{theorem}

The following immediate corollary of Theorem~\ref{T:ZK}, established in \cite{AFS},  will be used in Section~\ref{S:5} concerning the homogenization of a porous medium type equation.   

\begin{lemma}[cf.\ \cite{AFS}]\label{L:ort}
Let $\AA$ be an ergodic algebra in $\BUC(\R^n)$ and $h\in \BB^2$ such that 
$M(h\Delta f)=0$ for all $f\in \AA$ such that $\Delta f \in \AA$. Then $h$ is 
$\BB^2$-equivalent to a constant.
\end{lemma}

Theorem~\ref{T:ZK} also motivates the following definition.

\begin{definition}\label{D:new} An algebra w.m.v.\ $\AA$ is said to be {\em regular} if $Z(\AA)$ is dense in $V(\AA)$ in the topology of the $\sup$-norm.
\end{definition}

We have the following important fact about regular algebras w.m.v..

\begin{proposition}\label{P:reg2} If $\AA$ is a regular algebra w.m.v., then $\AA$ is ergodic.
\end{proposition}

\begin{proof} We are going to use the characterization of ergodic algebras provided by Lemma~\ref{L:ergo}. Let $f\in\AA$. Clearly, to prove \eqref{e.Lergo}, we may assume $M(f)=0$. Now, since $\AA$ is regular, given $\ve>0$, we may find $g\in Z(\AA)$ such that $\|f-g\|_\infty<\ve$. Hence,  
$$
\limsup_{t\to\infty}M_y\left(\bigl|\frac{1}{|B(0;t)|}
\int_{B(0;t)}f(x+y)\,dx\bigr|^2\right)\le 2\lim_{t\to\infty}M_y\left(\bigl|\frac{1}{|B(0;t)|}
\int_{B(0;t)}g(x+y)\,dx\bigr|^2\right)+2\ve^2=2\ve^2,
$$
where we used  Lemma~\ref{L:za}(ii) for the last equality. This implies \eqref{e.Lergo}.
\end{proof}

We next state a property of regular algebras~w.m.v.\ which will be used in our application to
homogenization of porous medium type equations on bounded domains in the final part of this paper. 

\begin{lemma}\label{L:1.3} Let $\AA$ be a regular algebra~w.m.v. If $f\in V(\AA)$,
then for any $\ve>0$ there exists a  function $u_\ve\in Z(\AA)$ satisfying the inequalities
\begin{gather}
f-\ve\leq \Delta u_\ve\le f+\ve.\label{e1.4_0}
\end{gather}
\end{lemma}
\begin{proof} This follows immediately from Lemma~\ref{L:za}(i) and Definition~\ref{D:new}.
\end{proof} 

The space $\FS(\R^n)$ studied in \cite{FS} provides a very encompassing example of a regular algebra w.m.v.. 

\begin{definition}\label{D:7} The Fourier-Stieltjes space, denoted by $\FSM(\R^n)$, is the completion relatively to the $\sup$-norm of the space of functions
 $\FSM_*(\R^n)$ defined by
\begin{equation}\label{schiara}
\FSM_*(\R^n):=\left\{f:\R^n\to\R\,:\,\text{$f(x)=\int_{\R^n} e^{i x\cdot y}\,d\nu(y)$ for some
$\nu\in\M_*(\R^n)$}\right\},
\end{equation}
where by $\M_*(\R^n)$ we denote the space of complex-valued measures $\mu$
with finite total variation, i.e., $|\mu|(\R^n)<\infty$.
\end{definition}

Recall that a subalgebra $\BB\subset\AA$ is called an ideal of $\AA$
if for any $f\in\AA$ and $g\in\BB$ we have $fg\in\BB$. Let
$C_0(\R^n)$ denote the closure of $C^\infty_c(\R^n)$ with respect to the
sup norm. The following result was established in~\cite{FS}.

\begin{proposition}[cf.\ \cite{FS}]\label{P:1} $\FSM(\R^n)\subset\BUC(\R^n)$ and it is an algebra w.m.v.\ containing $C_0(\R^n)$ as an ideal.
Moreover, $\FSM(\R^n)$ is a regular algebra w.m.v.\ and the space $\PAP(\R^n)$ of the perturbed almost periodic functions,
defined as
$$
\PAP(\R^n):=\{f\in\BUC(\R^n)\,:\,f=g+\psi,\ g\in\AP(\R^n),\ \psi\in C_0(\R^n)\},
$$
is a closed strict subalgebra of  $\FSM(\R^n)$.
\end{proposition}

\section{Two-scale Young Measures}\label{S:3}

In this section we recall the theorem giving the existence  of two-scale Young measures established in \cite{AFS}. We begin by recalling the concept of vector-valued
algebra with mean value.

Given a Banach space $E$ and an algebra w.m.v.\ $\AA$, 
we denote by $\AA(\re^n;E)$ the space 
of functions $f\in\BUC(\R^n;E)$ satisfying the following:
\begin{enumerate}
\item[(i)]  $L_f:=\la L,f\ra$ belongs to 
$\AA$ for all $L\in E^*$;

\item[(ii)] The family  $\{L_f\,:\, L\in E^*,\ \|L\|\le 1 \}$ is relatively
compact in $\AA$.
\end{enumerate}


\begin{theorem}[cf.\ \cite{AFS}]\label{T:T2}
Let $E$ be a Banach space, $\AA$ an algebra w.m.v.\ and ${\mathcal K}$ be the compact associated with
$\AA$. There is an isometric isomorphism between $\AA(\R^n;E)$ and
$C({\mathcal K};E)$. Denoting by $g\mapsto\util{g}$ the canonical map from
$\AA$ to $C({\mathcal K})$, the isomorphism associates to $f\in\AA(\R^n;E)$
the map $\tilde{f}\in C({\mathcal K};E)$ satisfying
\begin{equation}\label{mio2}
\util{\langle L,f\rangle}=\langle L,\tilde{f}\rangle\in C({\mathcal K})
\qquad\forall L\in E^*.
\end{equation}
In particular, for each $f\in \AA(\R^n;E)$, $\|f\|_E\in\AA$.
\end{theorem}

For $1\le p<\infty$, we define the space $L^p({\mathcal K};E)$ as the completion of $C({\mathcal K};E)$
with respect to the norm $\|\,\cdot\,\|_p$, defined as usual,
$$
\|f\|_p:=\left(\int_{{\mathcal K}}\|f\|_E^p\,d{\mathfrak m}\right)^{1/p}.
$$
As a standard procedure, we identify functions in $L^p$ that coincide ${\mathfrak m}$-a.e.\
in ${\mathcal K}$. 

Similarly, we define the space $\BB^p(\R^n; E)$ as the completion of $\AA(\R^n;E)$ with respect to the
seminorm
$$
|f|_p:=\left(\Medint_{\R^n}\|f\|_E^p\,dx\right)^{1/p},
$$
identifying functions in the same equivalence class determined by the seminorm $|\cdot|_p$. 
Clearly, the isometric isomorphism given by Theorem~\ref{T:T2} extends to an isometric isomorphism between
  $\BB^p(\R^n;E)$ and $L^p(\CK;E)$.

The next theorem gives the existence of two-scale Young measures associated 
with an algebra $\AA$. For the proof, we again refer to \cite{AFS}. 

Let $\Om\subset
\re^n$ be a bounded open set and $\{u_\ve(x)\}_{\ve>0}$ be a  family
of functions in $L^\infty(\Om;K)$, for some compact metric space $K$.

\begin{theorem}\label{T3}
Given any infinitesimal sequence $\{\ve_i\}_{i\in\N}$ there exist a
subnet $\{u_{\ve_{i(d)}}\}_{d\in D}$, indexed by a certain directed
set $D$, and a family of probability measures on $K$,
$\{\nu_{z,x}\}_{z\in {\mathcal K}, x\in \Om}$, weakly measurable with respect
to the product of the Borel $\sigma$-algebras in ${\mathcal K}$ and $\R^n$,
such that
\begin{equation}\label{young}
\lim_{D}\int_{\Om}\Phi(\frac{x}{\ve_{i(d)}},x,u_{\ve_{i(d)}}(x))\,dx=
\int_{\Om}\int_{{\mathcal K}}\la\nu_{z,x},\util{\Phi}(z,x,\cdot)\ra\,d{\mathfrak m}(z)\,dx
\qquad\forall\Phi\in\AA\left(\R^n;C_0(\Om\times K)\right).
\end{equation}
Here $\util{\Phi}\in C\left({\mathcal K};C_0(\Om\times K)\right)$ denotes the
unique extension of $\Phi$. Moreover, equality \eqref{young} still
holds for functions $\Phi$ in the following function spaces:
\begin{enumerate}
\item
$\BB^1(\R^n;C_0(\Om\times K))$;
\item
$\BB^p(\R^n;C(\bar \Om\times K))$ with $p>1$;
\item
$L^1(\Om;\AA(\R^n;C(K)))$.
\end{enumerate}
\end{theorem}

As in the  classical theory of Young measures  we have the following consequence
of Theorem~\ref{T3}.

\begin{lemma}\label{L:E} Let $\Om\subset\R^n$ be a bounded open set, let
$\{u_\ve\}\subset L^\infty(\Om;\R^m)$ be uniformly bounded
and let $\nu_{z,x}$ be a  two-scale Young
measure generated by a subnet $\{u_{\ve(d)}\}_{d\in D}$,
according to Theorem~\ref{T3}.  Assume that $U$ belongs
either to $L^1(\Om;\AA(\R^n;\R^m)))$ or to $\BB^p(\R^n;C(\bar \Om;\R^m))$ for some $p>1$.
Then
\begin{equation}\label{E:LE}
\nu_{z,x}=\d_{U(z,x)}\quad\text{if and only if}\quad
\lim_D\|u_{\ve(d)}(x)-U(\frac{x}{\ve(d)},x)\|_{L^1(\Om)}=0.
\end{equation}
\end{lemma}

In Sections~\ref{S:5} and~\ref{S:6}, we will need  a result similar to Theorem~\ref{L:E}, in which the corrector function $U(z,x)$ does not belong to either of the spaces in the statement. Namely, we will need the following result. 

\begin{lemma}\label{T:3.2} Let $\Om\subset\R^n$ be a bounded open set, let
$\{u_\ve\}\subset L^\infty(\Om)$ be uniformly bounded
and let $\nu_{z,x}$ be a  two-scale Young
measure generated by a subnet $\{u_{\ve(d)}\}_{d\in D}$,
according to Theorem~\ref{T3}. Let $U(z,x)=G(\theta(z,x))$ where $G:\R\to\R$ is a function in $\BV_\loc(\R)$, $\theta\in L^\infty(\Om;\AA(\R^n))$, and assume that 
\begin{equation}\label{eT:3.2} 
 \mm\left(\{z\in\KK\,:\, \theta(z,x)=\a\}\right)=0, 
 \end{equation}
 for a.e.\ $x\in\Om$, for all $\a\in E$, where $E$ is the set of discontinuity of $G$. Suppose, $\nu_{z,x}=\d_{U(z,x)}$. Then 
 \begin{equation}\label{eT:3.2.1}
\lim_D\|u_{\ve(d)}(x)-U(\frac{x}{\ve(d)},x)\|_{L^1(\Om)}=0.
\end{equation} 
 \end{lemma}
\begin{proof}  First, we observe that the values $\a\in\R$ for which \eqref{eT:3.2} does not hold, for a.e.\ $x\in\Om$, form a countable set. Indeed, since $\Om$ is bounded and 
$\mm(\KK)=1$,  for each $k\in\N$, there can be only a finite number of such values $\a\in\R$ for which
$$
\int_\Om  \mm\left(\{z\in\KK\,:\, \theta(z,x)=\a\}\right)\, dx>\frac{1}{k},
$$
and so the assertion follows. 

Since $G\in\BV_\loc(\R)$, the lateral limits $\lim_{s\to s_0\pm}G(s)$ exist, for all $s\in\R$, and, so, by \eqref{eT:3.2}, we may assume that $G$ is left-continuous. 
Also, by the properties of functions in $\BV_\loc(\R)$ ({\em cf.}, e.g., \cite{EG}), we know  that, in any compact interval of $I\subset \R$, $G$ may be written as  $G=G_1-G_2$ where $G_1$ and $G_2$ are monotone non-decreasing functions in $I$.  We may take $I$ such that $\theta(z,x)\in I$, for all $z\in\KK$, for a.e.\ $x\in\Om$. On the other hand, each $G_i$, $i=1,2$,
 is the uniform limit in $I$ of a monotone increasing sequence of piecewise constant nondecreasing functions.  The discontinuities of such piecewise constant functions may be suitably located at points $\a$ satisfying \eqref{eT:3.2}. Therefore, it suffices to prove the statement assuming that $G$ is such a  piecewise constant nondecreasing  function.  We may simplify further and consider $G$ as a piecewise constant function with only one discontinuity point, $\a_*$.  
 
 So, given $\d>0$, let us consider a continuous function $\z_\d:\R\to\R$, satisfying $0\le \z_\d\le 1$, $\z_\d(s)=1$, for $|s-\a_*|<\d$, and $\z_\d(s)=0$, for $|s-\a_*|\ge2\d$.  Let us denote $G_\d(s):=G(s)(1-\z_\d(s))$ and $U_\d(z,x):=G_\d(\theta(z,x))$. We have
$$
\lim_{\d\to 0}\int_{\Om}\int_{\KK} \z_\d(\theta(z,x))\,d\mm(z)\,dx=0,
$$
by the dominated convergence theorem, because of condition \eqref{eT:3.2}. Therefore, given $\g>0$, we may choose $\d_0>0$ sufficiently small such that, for $\d<\d_0$, 
$$
\lim_{\ve\to0}\int_{\Om} \z_\d(\theta(\frac{x}{\ve},x))\,dx=\int_{\Om}\int_{\KK}\z_\d(\theta(z,x))\,d\mm(z)\,dx<\g.
$$
Hence, 
$$
\limsup_{\ve(d)}\int_{\Om}|u_{\ve(d)}(x)-U(\frac{x}{\ve(d)},x)|\,dx\le \lim_{\ve(d)}\int_{\Om}|u_{\ve(d)}(x)-U_\d(\frac{x}{\ve(d)},x)|\,dx+\|G\|_\infty\g,
$$
and, by Theorem~\ref{T3},  since $\nu_{z,x}=\d_{U(z,x)}$, 
$$
 \lim_{\ve(d)}\int_{\Om}|u_{\ve(d)}(x)-U_\d(\frac{x}{\ve(d)},x)|\,dx=\int_{\Om}\int_{\KK}|U(z,x)-U_\d(z,x)|\,d\mm(z)\,dx<\|G\|_\infty\g,
 $$
 which gives
 $$
\limsup_{\ve(d)}\int_{\Om}|u_{\ve(d)}(x)-U(\frac{x}{\ve(d)},x)|\,dx\le2\g\|G\|_\infty.
$$
Since $\g>0$ is arbitrary, we arrive at \eqref{eT:3.2.1}, which finishes the proof. 
\end{proof}

\section{Some results about a porous medium type equation}\label{S:4}

In this section, we review some results about  an initial-boundary value problem for a porous medium type equation 
which will be used later. More specifically, let $\Om\subset \R^n$ be an open set, {\em possibly unbounded},  with smooth boundary we consider the following initial-boundary value problem
\begin{align}\label{e01}
& {\partial}_tu-\Delta f(x,u)=0, \qquad (x,t)\in Q:=\Om \times (0,+\infty),\\
& u(x,0)=u_{0}(x),\ x\in\Om, \label{e01'} \\
 & f(x,u(x,t))=0,\ (x,t)\in\po \Om\X(0,\infty). \label{e01''}
\end{align}

Concerning the initial data, we  assume 
\begin{equation}\label{eID}
u_0\in L^\infty(\Om). 
\end{equation}

For the purposes of this paper, we consider two types of functions $f(x,u)$ according 
to the following definitions.

\begin{definition}\label{D1'}
We say that the function $f(x,u)$ is  of  {\em type 1} if the conditions below are satisfied:
\begin{enumerate}
\item[{\bf(f1.1)}] $f:\bar\Om\X\R\to\R$ is continuous, for each $u\in\R$, $f(\cdot,u)$ is bounded and continuous in $\bar\Om$,  and, for each $x\in\bar\Om$, $f(x,\cdot):\re\to\re$ 
is strictly increasing and locally Lipschitz continuous uniformly in $x$. Moreover, $\lim_{u\to\pm\infty}f(x,u)=\pm\infty$, uniformly for $x\in\bar\Om$.

\item[{\bf(f1.2)}] $f(x,0)=0$ for $x\in\po\Om$. 
\end{enumerate}
\end{definition}

For the sake of example, observe that the assumptions {\bf(f1.1)} and {\bf(f1.2)} are trivially satisfied by functions of the form $f(x,u)=a(x)u|u|^{\g(x)}+b(x)$,  with  $\g,a,b$ smooth, bounded, $\g(x)>\g_0>0$,$a(x)>a_0>0$, $x\in\Om$, and $b(x)=0$,  for $x\in\po\Om$ .

We will also consider the problem \eqref{e01}-\eqref{e01''} when $f(x,u)$ is of the type described in the following definition.
\begin{definition}\label{D2'}
We say that the function $f(x,u)$ is  of  {\em type 2} if 
$f(x,u)=h(x)F(u)+S(x)$, where:
\begin{enumerate}
\item[{\bf(f2.1)}]  $F:\R\to\R$ is locally Lipschitz continuous, nondecreasing, $F(0)=0$, and $\lim_{u\to \pm \infty}F(u)=\pm\infty$. For definiteness, we assume that $F$ is {\em not} strictly increasing.  
\item[{\bf(f2.2)}] $S,h\in W^{2,\infty}(\Om)$, that is, they belong,  together with their derivatives up to second order, to $L^\infty(\Om)$, and 
$h(x)\ge \delta_0>0$, for all $x\in\Om$. 
\item[{\bf(f2.3)}] $S(x)=0$, for $x\in\po\Om$.
\end{enumerate}
\end{definition}

Observe that, for $F$ satisfying {\bf(f2.1)}   we may define $G(r)=\min\{u\,:\,F(u)=r\}$, and we have $F(G(r))=r$, for all $r\in\R$, and $G(F(u))=u$ if $u\notin F^{-1}(E)$, where 
$$
E:=\{ r\in \R\,:\, \text{$G$ is discontinuous at $r$}\}.
$$

\begin{remark} \label{R:F1} We remark that all the results obtained in what follows for $f(x,u)$ of type 2 have identical versions for $f(x,u)$ of the form $f(x,u)=F(h(x)u)+S(x)$, with $F,h,S$ satisfying the conditions in {\bf(f2.1)}, {\bf(f2.2)}, and {\bf(f2.3)}, the proofs of which  are easy adaptations of the proofs given herein for $f(x,u)$ of type 2, after the trivial change of variables $v=h(x)u$.   
\end{remark}

 \begin{remark}\label{R:F2} Concerning the homogeneous boundary condition in \eqref{e01''}, we remark that all discussion make in this section about this homogeneous problem can be immediately extended, with only minor adaptations, to apply  to the corresponding non-homogeneous problem formed replacing \eqref{e01''} by $f(x,u(x,t))=\beta(x)$, $(x,t)\in\po\Om\X(0,\infty)$, for any function $\b\in C(\bar\Om)\cap H_\loc^1(\bar\Om)$. Considering this more general boundary condition would have the convenient feature of allowing us to dispense with both assumptions {\bf(f1.2)} and {\bf(f2.3)}, which could be achieved in general by replacing a given $f(x,u)$ by another $\tilde f(x,u)=f(x,u)-\varphi(x)$, where $\varphi(x)$ is a harmonic function on $\Om$ that satisfies $\varphi(x)=f(x,0)$, for $x\in\po\Om$. 
 \end{remark}

In the case where  $f(x,u)$ is of type 1, since $f(x,\cdot)$ is (strictly) increasing, for each $x$,  then the equation \eqref{e01} is only mildly degenerate, in other words, it still belongs to the ``non-degenerate'' class, in the classification of \cite{CJ}. Nevertheless, it is degenerate in the sense that $f_u(x,\cdot)$ can vanish on a set ${\mathcal N}\subset \R$, provided ${\mathcal N}$ does not contain a non-empty open interval. The simplest and prototypical example is  the classical porous medium equation, for which $f(x,u)=u|u|^\g$, $\g>0$. We remark that for the latter, due to a comparison principle, we can always guarantee that $u(x,t)\ge0$ if $u_0(x)\ge0$,
which is physically desirable. For this reason, we can view $f(u)=u^{\g+1}$, $u\ge0$, as defined in $\R$, trivially extended as $u|u|^\g$. This motivates our choice of taking $f(x,\cdot)$ as defined in the whole $\R$, which is a matter of convenience. On the other hand, if $f(x,u)$ is of type 2, then the equation \eqref{e01} falls into the degenerate class  in the 
classification of \cite{CJ}.

The study of the well-posedness of the Cauchy problem for general quasilinear degenerate parabolic equations starts with Volpert and Hudjaev~\cite{VH}, for initial data in $BV$, where the $L^1$-stability was achieved completely only in the isotropic case, that is, for a diagonal viscosity matrix. The results in \cite{VH} were extended to the initial boundary value problem in \cite{WZ}. Well-posedness in the isotropic case with initial data in $L^\infty$ was established by Carrillo~\cite{CJ} in the homogeneous case where the coefficients do not explicitly depend on $(x,t)$. A purely $L^1$ well-posedness theory for the homogeneous anisotropic case was established by Chen and Perthame in \cite{CP}. The latter was extended to the non-homogeneous anisotropic case in~\cite{CK}.  We refer to the bibliography in the cited papers for a more complete list of references on the subject.

Equation \eqref{e01} is a particular case of a degenerate non-homegeneous isotropic equation and, as we said above, in the case where $f(x,u)$ is of type 1, its degeneration is of a mild type which makes its study a bit simpler than that of the general degenerate equation. On the other hand, in the case where $f(x,u)$ is of type 2, equation~\eqref{e01} is a particular case of a strongly degenerate parabolic equation. 
 Here we will review the  analysis of such equations for $f$ belonging to both types in order to introduce some notations and some particular results that will be needed in our study of the homogenization of porous medium type equations in Section~\ref{S:5}. For the stability results, in the type 1 case, we follow closely the analysis in \cite{CJ} and show which adaptations of the results in \cite{CJ} need to be made in order to handle the explicit dependence on $x$ of $f$. Still for the stability results, in the type 2 case, we borrow as well  some ideas from \cite{KO}, which in turn is also based on the analysis of \cite{CJ}.  
 
 For the existence of solutions, which follows from the compactness of the sequence of solutions of regularized (nondegenerate) problems, we introduce here a  method which is motivated by Kruzkhov~\cite{Kr}. We remark that recently Panov~\cite{Pa} has obtained a very general compactness result that, in particular, would imply the one proved here. However the techniques used in \cite{Pa} are out of the scope of the present paper and we think it is appropriate here to provide a simple and direct proof of this compactness result.

\begin{definition}\label{D1} A function $u\in L^\infty(Q)$ is said to be a weak solution of the problem \eqref{e01}-\eqref{e01''}, if the following hold:
\begin{enumerate}
\item[(1)] $f(x,u(x,t))\in L_\loc^2((0,\infty); H_{0,\loc}^1(\bar\Om))$;
\item[(2)] For any $\varphi\in C_c^\infty(\Om\X\R)$,  we have
\begin{equation}\label{e02}
\int\limits_{\Om\X(0,\infty)}\left\{ u(x,t)\varphi_t(x,t)-\nabla f(x,u(x,t))\cdot\nabla \varphi(x,t)\right\} \, dx\,dt+\int\limits_{\Om} u_0(x)\varphi(x,0)\,dx= 0.
\end{equation}
\end{enumerate}
\end{definition}

\begin{remark}\label{R:D1} In the case of a non-homogeneous boundary condition $f(x,u(x,t))=\beta(x)$, $(x,t)\in\po\Om\X(0,\infty)$, for some $\b\in C(\bar\Om)\cap H_\loc^1(\bar\Om)$, we only need to replace (1) in Definition~\ref{D1} by $f(x,u(x,t))-\b(x)\in L_\loc^2((0,\infty); H_{0,\loc}^1(\bar\Om))$.
\end{remark}

Let $u$ be a weak solution of  \eqref{e01}-\eqref{e01''}. Denoting by 
$\langle\cdot,\cdot\rangle$ the usual pairing between 
$H^{-1}(U)$ and $H_0^1(U)$ when $U\subset \re^n$ is open, we can 
conclude from \eqref{e02} that
$$
\partial_tu\in L_{\loc}^2((0,\infty);H_{\loc}^{-1}(\bar\Om)), 
$$
so that the equality \eqref{e02} is equivalent to 
\begin{equation}\label{e03}
\int_0^\infty\langle\partial_tu,\varphi\rangle\,dt+
\int_{Q}\nabla f(x,u)\cdot
\nabla\varphi\,dx\,dt-\int_{\Om} u_0\varphi(x,0)\,dx=0.
\end{equation}
for all $\varphi\in C_c^{\infty}(\Om\X\R)$.

Let $H_\d:\R\to\R$ be the approximation of the function $\sgn$  given by
$$
H_\d(s):=\begin{cases} 1, &\text{ for $s>\d$,}\\ \dfrac{s}\d, 
&\text{for $|s|\le \d$,}\\ -1, &\text{for $s<-\d$.}\end{cases}
$$

Given a nondecreasing Lipschitz continuous function $\vartheta:\re \to \re$ and $k\in \R$, we define

$$
B_{\vartheta}^k(x,\lambda):= \begin{cases} \int_k^{\lambda}\vartheta(f(x,r))dr, 
&\text{ if $f$ is of type $1$,}\\ 
\int_k^{\lambda}\vartheta(F(r))dr, &\text{if $f$ is of type $2$.}
\end{cases}
$$

Concerning the function $B_{\vartheta}^k$, we will make use of the following 
lemma which is a version of a lemma in~\cite{CJ}, whose proof remains 
essentially the same and for which, therefore, we refer to~\cite{CJ}, lemma~4, p.324.

\begin{lemma}\label{l01}
Let $u\in L^{\infty}(Q)$ be a weak solution 
of \eqref{e01}-\eqref{e01''}. Then, for a.e.\ $t\in (0,+\infty)$, we have
\begin{align*}
&\int_0^t \int_{\Om}B_{\vartheta}^k(x,u){\varphi}_s\,ds\,dx
+\int_{\Om}B_{\vartheta}^k(x,u_0)\varphi(x,0)\,dx
-\int_{\Om}B_{\vartheta}^k(x,u(t))\varphi(x,t)\, dx \\
&\qquad=-\int_0^t\langle{\partial}_su,\vartheta(f(x,u))\varphi\rangle\, ds
\end{align*}
$\forall k\in \re$ and $\forall \,0\le\varphi \in 
C_c^{\infty}(\Omega\times\R)$.
\end{lemma}

Let us denote
\begin{align*}
\vartheta^1_{\delta}(\lambda;y)&:=H_{\delta}(\lambda-f(y,k)),\\ 
 \vartheta^2_{\delta}(\lambda)&:=H_{\delta}(\lambda-F(k)),\quad\text{and}\\
B_{\vartheta^1_\d}^k(x,\l;y)&:= B_{\vartheta^1_\d(\cdot;y)}^k(x,\l).
\end{align*}

Next we state and prove a lemma which is also an adaptation of a similar result in \cite{CJ}, lemma~5, p.329.

\begin{lemma}[Entropy production term: type~1 case]\label{l02}
Let $u\in L^\infty(Q)$ be a weak solution of the problem~\eqref{e01}-\eqref{e01''}, 
with $u_0 \in L^{\infty}(\Om)$. If $f$ is of type 1, then
\begin{align}
&\int_{Q}B_{{\vartheta}^1_{\delta}}^k(x,u;y)\varphi_t
-H_{\delta}(f(x,u)-f(y,k))\nabla f(x,u)\cdot \nabla \varphi\,dx\,dt\label{eAT2}\\
&\qquad\qquad\qquad=\int_{Q}|\nabla f(x,u)|^2 H_{\delta}'(f(x,u)-f(y,k))\varphi \,dx\,dt,\nonumber
\end{align}
for all $y\in\Om$, $k\in \R$ and all $0\le \varphi \in C_c^{\infty}(Q)$.
\end{lemma}

\begin{proof}
By the Lemma~\ref{l01}, we have
$$-\int_0^{+\infty} \langle {\partial}_tu, H_{\delta}(f(x,u)-f(y,k))\varphi 
\rangle \,dt=\int_{Q}B_{{\vartheta}_{\delta}}^k(x,u;\mu) \varphi_t
\,dx\,dt.$$
Since $u$ is a weak solution and $H_{\delta}(f(x,u)-f(y,k))\varphi$ is a 
test function for each fixed $y$ and $k$, we get
\begin{equation*}
-\int_0^{+\infty}\langle {\partial}_tu,H_{\delta}(f(x,u)-f(y,k))
\varphi\rangle\,dt-\int_{Q}\{\nabla f(x,u)\cdot \nabla(H_{\delta}(f(x,u)-f(y,k))\varphi)\}\,dx\,dt=0.
\end{equation*}
This equality with the previous one gives
\begin{equation*}
\int_{Q}\{B_{{\vartheta}^1_{\delta}}^k(x,u;y)\varphi_t
-\nabla f(x,u)\cdot \nabla(H_{\delta}(f(x,u)-f(y,k))\varphi)\} \,dx\,dt=0,
\end{equation*}
and this yields \eqref{eAT2}.
\end{proof}

Now, we establish a result which is the analogue of Lemma~\ref{l02} for 
the case where $f$ is of type 2.

\begin{lemma}[Entropy production term: type  2 case]\label{l02'}
Let $u\in L^\infty(Q)$ be a weak solution of \eqref{e01}-\eqref{e01''} with 
$u_0\in L^{\infty}(\Om)$. If $f$ is of type $2$, then
\begin{align}
&\int_{Q}\big\{|u-k|\varphi_t-\nabla |f(x,u)-f(x,k)|\cdot \nabla 
\varphi - \sgn(u-k)\Delta f(x,k)\varphi\big\}\,dx\,dt\label{eAT2'}\\
&\qquad\qquad\qquad=\lim_{\delta\to 0}\int_{Q}h(x)|\nabla F(u)|^2 
H_{\delta}'(F(u)-F(k))\varphi \,dx\,dt,\nonumber
\end{align}
for all $k\in\R$ such that $F(k)\notin E$ and 
$0\le\varphi\in C^{\infty}_c(Q)$.
\end{lemma}

\begin{proof}
Similarly to what was done in Lemma~\ref{l02}, we have
$$
\int_{Q}\{B_{{\vartheta}^2_{\delta}}^k(u)\varphi_t
-\nabla f(x,u)\cdot \nabla(H_{\delta}(F(u)-F(k))\varphi)\} \,dx\,dt=0,
$$
which gives
\begin{align*}
&\int_{Q}\bigg\{ B_{{\vartheta}^2_{\delta}}^k(u)\varphi_t
-H_{\delta}(F(u)-F(k))\nabla \left(f(x,u)-f(x,k)\right)\cdot\nabla\varphi 
+H_{\delta}(F(u)-F(k))\Delta f(x,k)\,\varphi\bigg\}\,dx\,dt\\
&\qquad=\int_{Q}\bigg\{\left(F(u)-F(k)\right)\nabla h(x)\cdot 
\nabla F(u) +h(x)|\nabla F(u)|^2\bigg\} H_{\delta}'(F(u)-F(k))\varphi\,dx\,dt.
\end{align*}
Since $F(k)\notin E$, we obtain that $H_{\delta}(F(u)-F(k))\to \sgn(u-k)$ and 
$B_{{\vartheta}^2_{\delta}}^k(u)\to |u-k|$ as $\delta\to 0$. So, 
in order to obtain \eqref{eAT2'}, it suffices to show that 
the first integral on the right-hand side of the expression above goes to 
$0$ as $\delta\to 0$. For this, define
$$
I_\delta:= \int_{Q}\left(F(u)-F(k)\right)\nabla h(x)\cdot 
\nabla F(u) H_{\delta}'(F(u)-F(k))\varphi\,dx\,dt.
$$
A simple computation shows that
$$
I_\delta:= \int_{Q}\div\mathscr{F}_{\delta}(F(u))\varphi \,dx\,dt -\int_{Q}\Delta h\GG_\d(F(u))\varphi\,dx\,dt,
$$
where 
\begin{equation*}
\begin{aligned}
&\mathscr{F}_{\delta}(z):=\nabla h(x)\int_{F(k)}^z(r-F(k))
H_{\delta}'(r-F(k))\,dr,\\
&\GG_{\delta}(z):=\int_{F(k)}^z(r-F(k))
H_{\delta}'(r-F(k))\,dr.
\end{aligned}
\end{equation*}
Since $\lim_{\delta\to 0}\mathscr{F}_{\delta}(z)=0$ and  $\lim_{\delta\to 0}\GG_{\delta}(z)=0$   for all $z$, we have 
$\lim_{\delta\to 0} I_{\delta}=0$. 
\end{proof}

\begin{definition}\label{D2''} 
\begin{enumerate} 
\item[(i)] If $f(x,u)$ is of type~1, a function $u\in L^{\infty}(Q)$ is an {\em entropy solution} of the 
problem~\eqref{e01}-\eqref{e01''} if $u$ is just a weak solution of the same problem. 
\item[(ii)] If $f(x,u)$ is of type~2, a function $u\in L^{\infty}(Q)$ is an {\em entropy solution} of the 
problem~\eqref{e01}-\eqref{e01''} if $u$ is a weak solution and satisfies, 
for all $0\le \varphi\in C_c^\infty(\Om\X(0,\infty))$ and $k\in\R$, 
\begin{equation}\label{eD2''}
\int_Q\left\{ |u-k|\varphi_t- \nabla |f(x,u)-f(x,k)|\cdot \nabla \varphi -\sgn(u-k)\,\Delta f(x,k) \varphi\right\}\,dx\,dt \ge 0.
\end{equation}
\end{enumerate}
\end{definition}

The following theorem is a central tool in our analysis of the homogenization problem for porous medium type equation in Sections~\ref{S:5} and \ref{S:6}.  Its proof  follows from \eqref{eAT2}, by using doubling of 
variables, and the trick of completing the square in \cite{CJ}, theorem~13, p.~339.  Of particular importance for our homogenization study in Section~\ref{S:5} will be the formula
\eqref{eTA3'} below, which holds in the special case when one of the entropy solutions is stationary. 
We give the detailed proof here for the reader's convenience.

\begin{theorem}\label{T:A3} 
Let $u_1,u_2$ be entropy solutions  of the problem  \eqref{e01}-\eqref{e01''} with initial data $u_{01},u_{02}\in L^\infty(\Om)$. Then we have the following:
\begin{enumerate}
\item [(i)] For all $0\le \varphi\in C_c^\infty(Q)$, we have
\begin{equation}\label{eTA3}
\int_{Q} |u_1(x,t)-u_2(x,t)|\varphi_t-\nabla|f(x,u_1(x,t))-f(x,u_2(x,t))|
\cdot\nabla\varphi\,dx\,dt\ge 0.
\end{equation}

\item [(ii)] If $u_2$ is a stationary solution, then
\begin{align}\label{eTA3'}
&\int_{Q} |u_{1}(x,t)-u_2(x)|
\varphi_{t}-
\nabla|f(x,u_{1}(x,t))-f(x,u_2(x))|
\cdot\nabla\varphi \,dx\,dt\\
&\qquad\qquad=\lim_{\d\to0}\int_{Q}|\nabla[f(x,u_{1}(x,t))-
f(x,u_2(x))]|^2H_{\delta}'
(f(x,u_{1}(x,t))-f(x,u_2(x)))\varphi \,dx\,dt,\nonumber
\end{align}
for all $0\le \varphi \in C^{\infty}_c(Q)$.
\end{enumerate}

\end{theorem}

\begin{proof}
1.  In what follows,  we use the abridged notation  $u_1=u_1(x,t)$  and 
$u_2=u_2(y,s)$. We begin by proving \eqref{eTA3} in the case where $f$ is of type 1.  For this,  we apply  \eqref{eAT2}  to $u_1$,  to obtain
\begin{align*}
&\int_{Q}\{B_{{\vartheta}_{\delta}^1}^{k}(x,u_1;y)
{\phi}_t-H_{\delta}(f(x,u_1)-f(y,k)){\nabla}_xf(x,u_1)\cdot {\nabla}_x \phi \}\,dx\,dt 
\nonumber\\
&\qquad\qquad\qquad
=\int_{Q}|{\nabla}_xf(x,u_1)|^2H_{\delta}'(f(x,u_1)-f(y,k))
\phi\,dx\,dt,
\end{align*}
for all $k\in \re$ and for all $0\le \phi \in C_c^\infty(Q^2)$. 
Setting $k=u_2$ and integrating in $y,s$, we obtain
\begin{align}\label{E1}
&\int_{Q^2}\{B_{{\vartheta}_{\delta}^1}^{u_2}(x,u_1;y)
{\phi}_t 
-H_{\delta}(f(x,u_1)-f(y,u_2)){\nabla}_xf(x,u_1)\cdot {\nabla}_x \phi \,dx\,dt\,dy\,ds \nonumber\\
&=\int_{Q^2}
|{\nabla}_xf(x,u_1)|^2H_{\delta}'(f(x,u_1)-f(y,u_2))\phi\,dx\,dt\,dy\,ds.
\end{align}
Now, applying \eqref{eAT2} to $u_2$, taking $k=u_1$ and integrating in 
$x,t$, we obtain
\begin{align}\label{E2}
&&\int_{Q^2}\bigg{\{}B_{{\vartheta}_{\delta}^1}^{u_1}(y,u_2;x)
{\phi}_s+
H_{\delta}(f(x,u_1)-f(y,u_2)){\nabla}_yf(y,u_2)\cdot
{\nabla}_y\phi\bigg{\}}\,dx\,dt\,dy\,ds \nonumber\\
&&=\int_{Q^2}|{\nabla}_yf(y,u_2)|^2 
H_{\delta}'(f(x,u_1)-f(y,u_2))\phi\,dx\,dt\,dy\,ds
\end{align}
Now, we note that
\begin{align*}
0&=\int_{Q}{\nabla}_yf(y,u_2)\cdot {\nabla}_x[H_{\delta}
(f(x,u_1)-f(y,u_2))\phi]\,dx\,dt \nonumber\\
&=\int_{Q}\bigg{\{}{\nabla}_yf(y,u_2)\cdot {\nabla}_xf(x,u_1)H_{\delta}'(
f(x,u_1)-f(y,u_2))\phi+H_{\delta}(f(x,u_1)-f(y,u_2))
{\nabla}_yf(y,u_2)\cdot {\nabla}_x\phi\bigg{\}}\,dx\,dt
\end{align*}
and so we have
\begin{align}\label{E3}
&\int_{Q^2}H_{\delta}(f(x,u_1)-f(y,u_2))
{\nabla}_yf(y,u_2)\cdot 
{\nabla}_x\phi\,dx\,dt\,dy\,ds\nonumber\\ 
&
= -\int_{Q^2}{\nabla}_yf(y,u_2)\cdot {\nabla}_x
f(x,u_1)H_{\delta}'(f(x,u_1)-f(y,u_2))\phi\,dx\,dt\,dy\,ds
\end{align}
Analogously,
\begin{align}\label{E4}
&\int_{Q^2}H_{\delta}(f(x,u_1)-f(y,u_2))
{\nabla}_xf(x,u_1)\cdot{\nabla}_y\phi\,dx\,dt\,dy\,ds \nonumber\\
&
=\int_{Q^2}{\nabla}_yf(y,u_2)\cdot
{\nabla}_xf(x,u_1)H_{\delta}'(f(x,u_1)-f(y,u_2))\phi\,dx\,dt\,dy\,ds
\end{align}
Making  \eqref{E1} minus \eqref{E4} yields
\begin{align}\label{E5}
&\int_{Q^2}\bigg{\{}B_{{\vartheta}_{\delta}^1}^{u_2}(x,u_1;y)
{\phi}_t-H_{\delta}(f(x,u_1)-f(y,u_2)){\nabla}_xf(x,u_1)\cdot
({\nabla}_x+{\nabla}_y)\phi\bigg{\}}\,dx\,dt\,dy\,ds\nonumber\\
&=\int_{Q^2}\bigg{\{}|{\nabla}_xf(x,u_1)|^2
-{\nabla}_xf(x,u_1)\cdot {\nabla}_yf(y,u_2)\bigg{\}}H_{\delta}'(f(x,u_1)-f(y,u_2))\phi
\,dx\,dt\,dy\,ds
\end{align}
Further,  adding  \eqref{E2} and  \eqref{E3} gives
\begin{align}\label{E6}
&\int_{Q^2}\bigg{\{}B_{{\vartheta}_{\delta}^1}^{u_1}(y,u_2;x)
{\phi}_s+H_{\delta}(f(x,u_1)-f(y,u_2)){\nabla}_yf(y,u_2)\cdot
({\nabla}_x+{\nabla}_y)\phi\bigg{\}}\,dx\,dt\,dy\,ds\nonumber\\
&=\int_{Q^2}\bigg{\{}|{\nabla}_yf(y,u_2)|^2
-{\nabla}_xf(x,u_1)\cdot {\nabla}_yf(y,u_2)\bigg{\}}
H_{\delta}'(f(x,u_1)-f(y,u_2))\phi\,dx\,dt\,dy\,ds.
\end{align}
Now, adding \eqref{E5} and \eqref{E6} we obtain
\begin{align}\label{E7}
&\int_{Q^2}\bigg{\{}
B_{{\vartheta}_{\delta}^1}^{u_2}(x,u_1;y){\phi}_t+
B_{{\vartheta}_{\delta}^1}^{u_1}(y,u_2;x){\phi}_s\nonumber\\
&-H_{\delta}(f(x,u_1)-f(y,u_2))(\nabla_x+\nabla_y)(f(x,u_1)-f(y,u_2))\cdot(\nabla_x+\nabla_y)\phi\bigg{\}}\,dx\,dt\,dy\,ds\nonumber\\
&
=+\int_{Q^2}|(\nabla_x+\nabla_y)(f(x,u_1)-f(y,u_2))|^2
H_{\delta}'(f(x,u_1)-f(y,u_2))\phi\,dx\,dt\,dy\,ds.
\end{align}
We then use test functions as $\phi(x,t,y,s):= \varphi(\frac{x+y}{2},\frac{t+s}{2})
\rho_k (\frac{x-y}{2})\theta_l(\frac{t-s}{2})$, where $0\le \varphi \in C_c^{\infty}(Q)$, and $\rho_k,\ \theta_l$ are classical approximations of the identity in 
$\re^n$ and $\R$, respectively, as in the doubling of variables method. 
Hence, letting $k\to\infty$ first, later $\delta\to 0$ and then letting $l\to\infty$, we obtain~\eqref{eTA3} for $f$ of type 1.

2. Now, assume that the function $f$ is of type 2 and define the sets 
$$E_1:=\left\{(x,t)\in Q: F(u_1(x,t))\in E\right\}
\text{ and }
E_2:=\left\{(y,s)\in Q: F(u_2(y,s))\in E\right\}.$$ 
Observe that 
\begin{equation}\label{ch}
\sgn(u_1-u_2)=\sgn(F(u_1)-F(u_2)),
\end{equation}
for all $(x,t,y,s)\in\left\{\left(Q\setminus E_1\right)
\times Q\right\}
\cup\left\{Q\times \left(Q\setminus E_2\right)\right\}$.
Moreover,
\begin{equation}\label{ch1}
\nabla_xF(u_1)=0,\qquad\text{a.e.\ in $E_1$,}
\end{equation}
\begin{equation}\label{ch2}
\nabla_yF(u_2)=0,\qquad\text{a.e.\ in $E_2$.}
\end{equation}
Let $\phi$ be as in step 1. Using the Definition~\ref{D2''}, taking 
$k=u_2$ and integrating over $E_2$, we get
\begin{align}\label{Deg1}
&\int_{Q\times E_2}\big\{|u_1-u_2|\phi_t-\nabla_x
|f(x,u_1)-f(x,u_2)|\cdot\nabla_x\phi-\sgn(u_1-u_2)(\Delta f)(x,u_2)\phi 
\big\}\,dx\,dt\ge 0,
\end{align}
where $(\Delta f)(x,u):=\sum_{i=1}^nf_{x_ix_i}(x,u)$. Now, by applying 
Lemma~\ref{l02'} for $u_1$, taking $k=u_2(y,s)$ such that $(y,s)\notin E_2$, 
integrating over $Q\setminus E_2$ and adding to~\eqref{Deg1}, we 
have 
\begin{align}\label{Deg2}
&\int_{Q^2}\bigg\{|u_1-u_2|\phi_t-|F(u_1)-F(u_2)|
(\nabla h)(x)\cdot \nabla_x \phi -h(x)\nabla_x 
|F(u_1)-F(u_2)|\cdot\nabla_x\phi\nonumber\\
&\qquad\qquad\qquad-\sgn(u_1-u_2)(\Delta f)(x,u_2)\phi\bigg\}\,dx\,dt\,dy\,ds
\nonumber\\
&\qquad\qquad\ge 
\lim_{\delta\to 0}\int_{\left(Q\setminus E_1\right)\times 
\left(Q\setminus E_2\right)}h(x)|\nabla_xF(u_1)|^2 
H_{\delta}'(F(u_1)-F(u_2))\phi\,dx\,dt\,dy\,ds.
\end{align}
By arguing in a similar way for $u_2$ we can prove that 
\begin{align}\label{Deg3}
&\int_{Q^2}\bigg\{|u_1-u_2|\phi_s-|F(u_1)-F(u_2)|
(\nabla h)(y)\cdot \nabla_y \phi -h(y)\nabla_y 
|F(u_1)-F(u_2)|\cdot\nabla_y\phi\nonumber\\
&\qquad\qquad\qquad+\sgn(u_1-u_2)(\Delta f)(y,u_1)\phi\bigg\}\,dx\,dt\,dy\,ds
\nonumber\\
&\qquad\qquad\ge 
\lim_{\delta\to 0}\int_{\left(Q\setminus E_1\right)\times 
\left(Q\setminus E_2\right)}h(y)|\nabla_yF(u_2)|^2 
H_{\delta}'(F(u_1)-F(u_2))\phi\,dx\,dt\,dy\,ds.
\end{align}

3. Since 
$$
0=\int_{Q}h(y)\nabla_yF(u_2)\cdot \nabla_x
\left(H_{\delta}(F(u_1)-F(u_2))\phi\right)\,dx\,dt,
$$
we obtain, taking into account~\eqref{ch}--\eqref{ch2}
\begin{align}\label{Deg4}
&\int_{Q^2}h(y)\nabla_y|F(u_1)-F(u_2)|\cdot \nabla_x\phi 
\,dx\,dt\,dy\,ds\nonumber\\
&\qquad=\lim_{\delta\to 0}\int_{\left(Q\setminus E_1\right)\times 
\left(Q\setminus E_2\right)} 
h(y)\nabla_yF(u_2)\cdot\nabla_x 
F(u_2)H_{\delta}'(F(u_1)-F(u_2))\phi\,dx\,dt\,dy\,ds.
\end{align}
Analogously,
\begin{align}\label{Deg5}
&\int_{Q^2}h(x)\nabla_x|F(u_1)-F(u_2)|\cdot \nabla_y\phi 
\,dx\,dt\,dy\,ds\nonumber\\
&\qquad=\lim_{\delta\to 0}\int_{\left(Q\setminus E_1\right)\times 
\left(Q\setminus E_2\right)} 
h(x)\nabla_yF(u_2)\cdot\nabla_x 
F(u_2)H_{\delta}'(F(u_1)-F(u_2))\phi\,dx\,dt\,dy\,ds.
\end{align}

4. Multiplying~\eqref{Deg5} by $-1$ and adding to~\eqref{Deg2}, we get
\begin{align}\label{Deg6}
&\int_{Q^2}\bigg\{|u_1-u_2|\phi_t-|F(u_1)-F(u_2)|
(\nabla h)(x)\cdot \nabla_x \phi -h(x)\nabla_x 
|F(u_1)-F(u_2)|\cdot\left(\nabla_x+\nabla_y\right)\phi\nonumber\\
&\qquad\qquad\qquad\qquad\qquad\qquad\qquad
-\sgn(u_1-u_2)(\Delta f)(x,u_2)\phi\bigg\}\,dx\,dt\,dy\,ds
\nonumber\\
&\qquad\qquad\ge 
\lim_{\delta\to 0}\int_{\left(Q\setminus E_1\right)\times 
\left(Q\setminus E_2\right)}\bigg\{h(x)|\nabla_xF(u_1)|^2
\nonumber\\
&\qquad\qquad\qquad\qquad\qquad-h(x)\nabla_xF(u_1)\cdot\nabla_yF(u_2)\bigg\}
H_{\delta}'(F(u_1)-F(u_2))\phi\,dx\,dt\,dy\,ds.
\end{align}
Similarly with respect to~\eqref{Deg4} and~\eqref{Deg3},
\begin{align}\label{Deg7}
&\int_{Q^2}\bigg\{|u_1-u_2|\phi_s-|F(u_1)-F(u_2)|
(\nabla h)(y)\cdot \nabla_y \phi -h(y)\nabla_y 
|F(u_1)-F(u_2)|\cdot\left(\nabla_x+\nabla_y\right)\phi\nonumber\\
&\qquad\qquad\qquad+\sgn(u_1-u_2)(\Delta f)(y,u_1)\phi\bigg\}\,dx\,dt\,dy\,ds
\nonumber\\
&\qquad\qquad\ge 
\lim_{\delta\to 0}\int_{\left(Q\setminus E_1\right)\times 
\left(Q\setminus E_2\right)}\bigg\{h(y)|\nabla_yF(u_2)|^2
\nonumber\\
&\qquad\qquad\qquad\qquad\qquad-h(y)\nabla_xF(u_1)\cdot\nabla_yF(u_2)\bigg\}
H_{\delta}'(F(u_1)-F(u_2))\phi\,dx\,dt\,dy\,ds.
\end{align}
Finally, adding the last two inequalities yields 
\begin{align*}
&\int_{Q^2}\bigg\{|u_1-u_2|\left(\phi_t+\phi_s\right)
-|F(u_1)-F(u_2)|\bigg((\nabla h)(x)\cdot \nabla_x\phi+(\nabla h)(y)
\cdot \nabla_y\phi\bigg)\\
&\qquad\qquad-\bigg(h(x)\,\nabla_x|F(u_1)-F(u_2)|+h(y)\,\nabla_y|F(u_1)-F(u_2)|
\bigg)\cdot\big(\nabla_x+\nabla_y\big)\phi\\
&\qquad\qquad\qquad\qquad\qquad-\sgn(u_1-u_2)\bigg((\Delta f)(x,u_2)-
(\Delta f)(y,u_1)\bigg)\phi\bigg\}\,dx\,dt\,dy\,ds\\
&\qquad\qquad\ge \lim_{\delta\to 0}
\int_{\left(Q\setminus E_1\right)\times 
\left(Q\setminus E_2\right)}
\bigg\{ \big|h(x)\,\nabla_xF(u_1)-h(y)\,\nabla_yF(u_2)\big|^2\\
&\qquad\qquad\qquad\qquad +\big(h(x)-h(y)\big)^2\,\nabla_xF(u_1)\cdot 
\nabla_yF(u_2)\bigg\}H_{\delta}'(F(u_1)-F(u_2))\phi\,dx\,dt\,dy\,ds,
\end{align*}
which is equivalent to 
\begin{align}\label{Deg8}
&\int_{Q^2}\bigg\{|u_1-u_2|\left(\phi_t+\phi_s\right)-
\big(\nabla_x+\nabla_y\big)|f(y,u_1)-f(y,u_2)|\cdot 
\big(\nabla_x+\nabla_y\big)\phi\nonumber\\
&\qquad\qquad\qquad\qquad -\sgn(u_1-u_2)\bigg((\Delta f)(x,u_2)-
(\Delta f)(y,u_1)\bigg)\phi\bigg\}\,dx\,dt\,dy\,ds\nonumber\\
&\qquad\qquad\ge\lim_{\delta\to 0}\int_{Q^2}\bigg\{
\big(h(x)-h(y)\big)^2\,\nabla_xF(u_1)\cdot 
\nabla_yF(u_2)H_{\delta}'(F(u_1)-F(u_2))\phi
\nonumber\\
&\qquad\qquad\qquad\qquad\qquad\qquad\qquad
-\big(h(x)-h(y)\big)\,\nabla_x|F(u_1)-F(u_2)|
\cdot\big(\nabla_x+\nabla_y\big)\phi
\nonumber\\
&\qquad\qquad\qquad\qquad\qquad\qquad\qquad\qquad\qquad
+|F(u_1)-F(u_2)|\,\bigg((\nabla h)(x)-(\nabla h)(y)
\bigg)\cdot \nabla_x \phi\bigg\}\,dx\,dt\,dy\,ds\nonumber\\
&\qquad\qquad\qquad=\lim_{\delta\to 0}\big(I_1^{\delta}+I_2+I_3\big).
\end{align}
Now, observe that 
\begin{align*}
&I_1^{\delta}=\int_{Q^2}\big(h(x)-h(y)\big)^2\,\nabla_xF(u_1)
\cdot\nabla_yF(u_2)H_{\delta}'(F(u_1)-F(u_2))\phi\,dx\,dt\,dy\,ds
\\
&\qquad=\int_{Q^2}\big(h(x)-h(y)\big)^2\,
\nabla_yF(u_2)\cdot \nabla_x\big(H_{\delta}(F(u_1)-F(u_2))\big)\phi
\,dx\,dt\,dy\,ds\\
&\qquad=-\int_{Q^2}H_{\delta}(F(u_1)-F(u_2))\nabla_yF(u_2)\cdot
\bigg(\nabla_x\phi\,\big(h(x)-h(y)\big)^2+2(\nabla h)(x)\,
\big(h(x)-h(y)\big)\phi\bigg)\,dx\,dt\,dy\,ds\\
&\qquad\le C\int_{Q^2}|\nabla_yF(u_2)|\,|x-y|
\bigg(|x-y|\,|\nabla_x\phi|+2|(\nabla h)(x)|\phi\bigg)\,dx\,dt\,dy\,ds.
\end{align*}
Taking $\phi(x,t,y,s):= \varphi(\frac{x+y}{2},\frac{t+s}{2})
\rho_k (\frac{x-y}{2})\theta_l(\frac{t-s}{2})$ as in the step 1, the previous 
inequality shows that $I_1^{\delta}\to 0$ when $k\to \infty$ uniformly in 
$\delta$. Similarly, we can prove that $I_2\to 0$ as $k\to \infty$. Moreover, 
\begin{align*}
&I_3=-\int_{Q^2}\bigg\{\nabla_x|F(u_1)-F(u_2)|\cdot \big(
(\nabla h)(x)-(\nabla h)(y)\big)\phi +|F(u_1)-F(u_2)|\,(\Delta h)(x)\phi 
\bigg\}\,dx\,dt\,dy\,ds,
\end{align*}
where, like above, the first integral goes to $0$ as $k\to\infty$ and it is
easy to check that the second one goes to 
$$-\int_{Q}\sgn(u_1-u_2)\bigg((\Delta f)(x,u_2(x,t))-
(\Delta f)(x,u_1(x,t))\bigg)\varphi(x,t)\,dx\,dt,
$$
as $k,l\to\infty$. Finally, using this facts and taking $k,l\to\infty$ 
in~\eqref{Deg8}, we obtain~\eqref{eTA3} for $f$  of type 2.

5. To obtain~\eqref{eTA3'}, we observe that if $u_2$ is stationary solution 
then $B_{{\vartheta}_{\delta}}^{u_1}(y,u_2;x)$ and $B_{{\vartheta}_{\delta}}^{u_2}(x,u_1;y)$ are independent of $s$ and so, 
we can write the trivial equality where both members are null
$$
\int_{Q^2}B_{{\vartheta}_{\delta}}^{u_1}(y,u_2;x)
{\phi}_s\,dx\,dt\,dy\,ds
=\int_{Q^2}B_{{\vartheta}_{\delta}}^{u_2}(x,u_1;y)
{\phi}_s\,dx\,dt\,dy\,ds
$$
Combining the previous equality in~\eqref{E7}, we have
\begin{align*}
&\int_{Q^2}\bigg{\{}
B_{{\vartheta}_{\delta}}^{u_2}(x,u_1;y)({\phi}_t+{\phi}_s)\nonumber\\
&-H_{\delta}(f(x,u_1)-f(y,u_2))(\nabla_x+\nabla_y)(f(x,u_1)-f(y,u_2))
\cdot(\nabla_x+\nabla_y)\phi\bigg\}\,dx\,dy\,dt\,ds\nonumber\\
&=\int_{Q^2}|(\nabla_x+\nabla_y)(f(x,u_1)-f(y,u_2))|^2
H_{\delta}'(f(x,u_1)-f(y,u_2))\phi\,dx\,dt\,dy\,ds.
\end{align*}
Now, using test functions as above and letting $k,l\to\infty$, we get \eqref{eTA3'}.

\end{proof}

\begin{remark}\label{obs1}
As usual, we denote $(s)_\pm:=\max\{\pm s,0\}$. 
The same arguments in the above proof lead to an inequality similar to \eqref{eTA3}, with $|u_1-u_2|$, $|f(x,u_1)-f(x,u_2)|$ replaced by $(u_1-u_2)_\pm$, $(f(x,u_1)-f(x,u_2))_\pm$, respectively,  just by using  $B_{({\vartheta}_{\delta})_\pm}^k$, $(H_\d)_\pm$, instead of $B_{{\vartheta}_{\delta}}^k$, $H_\d$, respectively. We thus obtain
\begin{equation}\label{eTA3''}
\int_{Q} (u_1(x,t)-u_2(x,t))_\pm\varphi_t-\nabla(f(x,u_1(x,t))-f(x,u_2(x,t)))_\pm
\cdot\nabla\varphi\,dx\,dt\ge 0.
\end{equation}
 where we mean one inequality holding with $(\cdot)_+$ and another holding for $(\cdot)_-$.
Moreover,  to obtain  \eqref{eTA3''} we only need that $u_i\in L^\infty(Q)$ satisfies \eqref{e02}, if $f$ is of type~1, or \eqref{eD2''}, if $f$ is of type~2, and  $f(x,u_i(x,t))\in L_\loc^2((0,\infty);H_\loc^1(\bar\Om))$
instead of $f(x,u_i(x,t))\in L_\loc^2((0,\infty);H_{0,\loc}^1(\bar\Om))$, $i=1,2$, as can be easily checked.

\end{remark}

Given any $R>0$,  let $\xi_R\in H^1_0(\bar\Omega\cap B(0;R))$ be the eigenfunction of $-\Delta$ 
associated with the eigenvalue $\lambda_1(R)>0$ such that $\xi_R>0$ in $\Omega\cap B(0;R)$ (see, e.g., \cite{Ev}). 

\begin{theorem}[Uniqueness]\label{T:A4}
Let $u_1,u_2$ be  entropy solutions of the problem  \eqref{e01}-\eqref{e01''} with initial data $u_{01},u_{02}\in L^\infty(\Om)$. 
Then, for each $R>0$, there exists $C>0$, such that  for a.e. $t>0$, we have
\begin{equation}\label{eTA4}
\int_{\Omega}|u_1(t)-u_2(t)|\xi_R(x)\,dx\le 
e^{Ct}\int_{\Omega}|u_{01}(x)-u_{02}(x)|\xi_R(x)\,dx.
\end{equation}
\end{theorem}

\begin{proof} Taking $\varphi(x,t)=\delta_h(t)\xi_R(x)$, with $0\le \delta_h\in
C^{\infty}_c((0,+\infty))$ in (i) of Theorem~\eqref{T:A3}, we obtain
\begin{eqnarray*}
&&\int_{Q}\bigg{\{}-|u_1-u_2|\delta_h'(t)\xi_R(x)
-|f(x,u_1)-f(x,u_2)|\delta_h(t)\Delta\xi_R(x) \bigg{\}}\,dx\,dt\le 0.
\end{eqnarray*}
Observe that 
\begin{align*}
-\int_{Q}|u_1-u_2|\delta_h'(t)\xi_R(x)\,dx\,dt&\le
\int_{Q}\bigg{\{}|f(x,u_1)-f(x,u_2)|\delta_h(t)|\Delta\xi_R(x)|\bigg{\}}\,dx\,dt\\
&\le C\int_{Q}|u_1-u_2|\delta_h(t)
\xi_R(x)\,dx\,dt,
\end{align*}
where we use that $|\Delta\xi_R|=\l_1\xi_R$ and the Lipschitz 
condition on $f(x,u)$. 
We define 
$$
\beta(s):=\int_{\Om}|u_1(x,s)-u_2(x,s)|\xi_R(x)\,dx.
$$
Then, using a suitable sequence of functions $\delta_h$ and letting $h\to0$, we arrive at
\begin{eqnarray*}
\beta(t)\le \int_{\Om}|u_{01}(x)-u_{02}(x)|\xi_R(x)\,dx+C\int_0^t\beta(s)\,ds.
\end{eqnarray*}
Hence, we may apply Gronwall's lemma to conclude the proof of \eqref{eTA4}.
\end{proof}

\begin{remark}\label{obs2}
Noting that $(f(x,u_1)-f(x,u_2))_\pm\le C(u_1-u_2)_\pm$, respectively, and using Remark~\ref{obs1} we see that the same arguments show that 
\begin{equation}\label{eTA4'}
\int_{\Om}(u_1(t)-u_2(t))_\pm\xi_R(x)\,dx\le 
e^{Ct}\int_{\Om}(u_{01}(x)-u_{02}(x))_\pm \xi_R(x)\, dx
\end{equation}
for a.e.\ $t>0$ for entropy solutions of the problem \eqref{e01}-\eqref{e01''}. 
Moreover, as a consequence of Remark~\ref{obs1}, to obtain  \eqref{eTA4'} 
we only need that $u_i\in L^\infty(Q)$ satisfies \eqref{e02} and  $f(x,u_i(x,t))\in L_\loc^2((0,\infty);H_\loc^1(\bar\Om))$
instead of $f(x,u_i(x,t))\in L_\loc^2((0,\infty);H_{0,\loc}^1(\bar\Om))$, $i=1,2$, provided
\begin{equation}\label{eobs2}
(f(x,u_1(x,t))-f(x,u_2(x,t)))_\pm|\po\Om\equiv 0, \qquad a.e.\ t\in(0,\infty), \text{respectively}, 
\end{equation}
the latter meaning the trace on $\po\Om$ for functions in $H_\loc^1(\bar\Om)$.
\end{remark}

The above remark immediately implies the following result.

\begin{corollary}[Monotonicity]\label{cor1}
Let $u_1,u_2\in L^\infty(Q)$ satisfy \eqref{e02}, if $f$ is of type~1, or \eqref{eD2''}, if $f$ is of type~2,  and, in either case,  $f(x,u_i(x,t))\in L_\loc^2((0,\infty);H_\loc^1(\bar\Om))$, $i=1,2$. Suppose that 
$u_{01}(x)\le u_{02}(x)$  for a.e.\ $x\in\Om$ and
\begin{equation}\label{ecor1}
(f(x,u_1(x,t))-f(x,u_2(x,t)))_+|\po\Om\equiv 0, \qquad a.e.\ t\in(0,\infty). 
\end{equation}
Then,
$$
u_1(x,t)\le u_2(x,t),\qquad \text{for a.e.\ $(x,t)\in Q$}.
$$
\end{corollary}

\bigskip 

Our next goal is to prove the existence of an entropy solution for \eqref{e01}-\eqref{e01''} .

We consider the following regularized version of \eqref{e01}-\eqref{e01''},
\begin{align}\label{e04}
&{\partial}_t{u}-\Delta f^{\s}(x,u)=0, \qquad (x,t)\in Q,\\
&u(x,0)=u_{0,\s}(x),\  x\in \Om, \label{e04'}\\
&u(x,t)=0,\ \text{for $(x,t)\in \po\Om\X(0,\infty)$} \label{e04''}
\end{align} 
where
$$
f^{\s}(x,u):=(\rho_\s^{(n+1)}*f)(x,u)-(\rho_\s^{(n+1)}*f)(x,0)+ \tilde f_\s(x) +\s u, 
$$
where $\rho_\s^{(n+1)}(x,u)=\rho_\s(x_1)\cdots\rho_\s(x_n)\rho_\s(u)$, where $\rho_\s(s)$ is a standard Dirac  sequence of mollifiers in $\R$,  we assume $f(x,u)$ extended by 0 outside $\Om\X\R$, and 
$\tilde f_\s(x)=[\rho_\s^{(n)}*(\chi_{{}_{\Om_\s}} f(\cdot,0))](x)$, where $\rho_\s^{(n)}(x)=\rho_\s(x_1)\cdots\rho_\s(x_n)$ and  $\chi_{{}_{\Om_\s}}$ is the characteristic function of the set $\Om_\s:=\{x\in\Om\,:\, \operatorname{dist}(x;\po\Om)>\s\}$.
We also prescribe a regularized  initial data 
\begin{equation}\label{eu0}
u_{0,\s}:=\rho_\s^{(n)}* (\chi_{{}_{\Om_\s}}\, u_0).
\end{equation}

The existence and uniqueness of a classical solution of \eqref{e04},\eqref{e04'},\eqref{e04''},  for $\s>0$, with $u_{0,\s}$ defined by \eqref{eu0}, is proved, for example, in~\cite{LSU}. 

Following Kruzhkov's ideas in \cite{Kr}, we now establish the following result, which gives the pre-compactness in $L_{\loc}^1(Q)$ of the classical solutions $u_\s$, 
when $f$ is of type~2. It will be convenient to use again $\xi_R$, defined just before Theorem~\ref{T:A4}.  

\begin{theorem}\label{l04}
Assume $f(x,u)$ be of type~2 and $u_0\in W^{1,\infty}(\Om)$. Let $u_{\s}$ be the solution of the 
regularized problem \eqref{e04},\eqref{e04'},\eqref{e04''}. 
Then,  
\begin{equation}\label{eTcomp1}
\|u_\s(t)\|_{L^\infty(\Om)}\le M_0,\quad t\ge0,
\end{equation}
with $M_0$ independent of $\s$, and, for any $R>0$, $T>0$, and $|y|<\d$, with  $\d$ sufficiently small,
\begin{equation}\label{e05}
\int_{\Om}|u_{\s}(x+y,t)-u_{\s}(x,t)|\xi_R(x)\,dx\le  c_1 \delta, \quad 0\le t\le T,
\end{equation}
where the constant $c_1=c_1(R,T,\|\nabla u_0\|_{\infty},\sum_{i=1}^n \|f_{x_ix_i}\|_{L^\infty(\Om\X I)} )$ is independent of $\s$, with $I=[-M_0,M_0]$. Moreover, for some constant $M>0$ independent of $\s$, for any $R>0$, $0\le t<T$, we have
\begin{equation}\label{e06}
\int_{\Om}|u_{\s}(x,t+s)-u_{\s}(x,t)|\xi_R(x)\,dx\le
\min_{0<\delta<1}\bigg\{(2c_1+\|\nabla\xi_R\|_1)\delta +s\,M\bigg(\frac{1}{\delta^2}
+\frac{2}{\delta}+1\bigg)\|\xi_R\|_1 \bigg\} \stackrel{s\to 0}{\longrightarrow} 0.
\end{equation}

\end{theorem}

\begin{proof}  The uniform boundedness stated in \eqref{eTcomp1} is obtained by direct application of Corollary~\ref{cor1}, and Remark~\ref{obs2},   by comparing $u_\s$ with the stationary solutions $g_\s(x,\pm M)$, where $g_\s(x,f_\s(x,\a))=\a$,  since we may take $M>0$ large enough so that $g_\s(x,-M)\le u_0(x)\le g_\s(x,M)$,  $\|g_\s(x,\pm M)\|_\infty\le M_0$, for some $M_0>0$ independent of $\s$, and $g_\s(x,-M)\le 0\le g_\s(x,M)$, for $x\in\po\Om$.  

1. To prove~\eqref{e05}, for each $k=1,\cdots,n$ define $v^k:=\partial_{x_k}u_{\s}$ and observe that
\begin{equation}\label{kr1}
\partial_tv^k-\Delta (f^{\s}_u(x,u)v^k)
-\nabla\cdot (f^{\s}_{x_ku}(x,u)\nabla u)-
\big(\nabla f^{\s}_{x_ku}\big)(x,u)\cdot \nabla u=
-\big(\Delta f^\s\big)(x,u),
\end{equation}
where, for simplicity of notation,  we denote $u_{\s}$ by $u$, 
$\big(f^{\s}_{x_1x_ku}(x,u),\cdots, f^{\s}_{x_nx_ku}(x,u)\big)$ by 
$\big(\nabla f^{\s}_{x_ku}\big)(x,u)$ and 
$\sum_{i=1}^nf^{\s}_{x_ix_i}(x,u)$ by $\big(\Delta f^\s\big)(x,u)$.

We fix a number $T>0$ and let $g^k\in C^{\infty}(\Om\X [0,T])$ be such 
that $g^k(t)\in C_c^\infty(\Om)$ for all $t\in [0,T]$. Now, 
taking $0<t_0\le T$, multiplying 
the equation~\eqref{kr1} by $g^k$, integrating by parts and summing over 
$k$ from $1$ to $n$, we get

\begin{align}\label{kr2}
&\int_0^{t_0}\int_{\Om}-\sum_{k=1}^n\bigg{\{}\partial_tg^k+
f^\s_u(x,u)\Delta g^k-\sum_{i=1}^n\big(f^{\s}_{x_iu}(x,u)g^i_{x_k}
-f^{\s}_{x_ix_ku}(x,u)g^i\big)\bigg{\}}v^k\,dx\,dt\nonumber\\
&\qquad+\int_{\Om}\sum_{k=1}^nv^k(t_0)g^k(t_0)\,dx=
\int_{\Om}\sum_{k=1}^n\bigg{\{}v^k(0)g^k(0)-(\Delta f)(x,u)g^k(t_0)\bigg{\}}\,dx.
\end{align}
For $k=1,\cdots,n$ and $g=(g^1,\cdots,g^n)$, we define
\begin{equation}\label{kr3}
\mathscr{L}_k(g):=\partial_tg^k+f^{\s}_u(x,u)\Delta g^k-
\sum_{i=1}^n\big(g^i_{x_k}f^{\s}_{x_iu}(x,u)-f^{\s}_{x_ix_ku}(x,u)
g^i\big).
\end{equation}
Now we define  $\varphi^k_h$, $k=1,\cdots,n$, as the solution of the (backward) initial-boundary value problem 
\begin{equation}\label{kr4'}
\begin{cases}
\mathscr{L}_k(\varphi_h)= 0,&(x,t)\in \Om\times(0,t_0),\\ 
\varphi^k_h(t_0)=\big(\sgn (v^k(t_0))\chi_{{}_{\Om_{2h}}}\big)\ast \rho_h\, e^{-|x|} ,& x\in \Om,\\
\varphi_h^k(x,t)=0, & (x,t)\in\po\Om\X(0,t_0),
\end{cases}
\end{equation}
where $\chi_{{}_A}$ denotes, as usual, the indicator function of the set $A$, and $\Om_{2h}:=\{\,x\in\Om\,:\,\dist(x,\po\Om)>2h\,\}$, and 
 $\rho_h=h^{-n}\rho(h^{-1}x)$, and $0\le \rho\in C_c(\R^n)$ is a standard symmetric mollifier satisfying  $\supp \rho\subset\{x\,:\,|x|\le 1\}$ and $\int_{\R^n}\rho\,dx=1$.

Now, observe that 
\begin{align*}
&0=2\mathscr{L}_k(\varphi_h)\varphi^k_h=
\partial_t(\varphi^k_h)^2+f^{\s}_u(x,u)\Delta (\varphi^k_h)^2-
2f^{\s}_u(x,u)|\nabla \varphi^k_h|^2\\
&\qquad\qquad
-2\sum_{i=1}^nf^{\s}_{x_iu}(x,u)\varphi^i_{h,x_k}\varphi^k_h
+2\sum_{i=i}^nf^{\s}_{x_ix_ku}(x,u)\varphi^i_h\varphi^k_h
\end{align*}

Since  $f$ is of type~2,  clearly, for $\gamma_0$ sufficiently small,
$$
f_u^{\s}(x,u)-\gamma_0\sum_{i=1}^n |f_{x_i u}^\s(x,u)|\ge 0,
$$
for all $(x,u)\in\bar\Om\X\R$.
Therefore, using Cauchy inequality and  summing over $k$,  we arrive at
\begin{align}
&0\le \partial_t |\varphi_h|^2+f^{\s}_u(x,u)\Delta |\varphi_h|^2
+2(-f_u^\s(x,u)+\gamma_0\sum_{i=1}^n |f_{x_i u}^\s(x,u)|)\sum_{k=1}^n|\nabla \varphi^k_h|^2 +c(\gamma_0)|\varphi_h|^2              \label{ephi}\\
&\qquad\le\partial_t|\varphi_h|^2+f^{\s}_u(x,u)\Delta|\varphi_h|^2 +c|\varphi_h|^2. \nonumber
\end{align}


2. In this step, we prove that 
$$
|\varphi_h|^2\le c(T)\,e^{-\frac{|x|}{M}},
$$
for all $(x,t)\in \Om\times [0,t_0]$.

We begin by defining $\mathscr{L}(v):= \partial_tv+f^{\s}_u(x,u)\Delta v$,  
$w:=|\varphi_h|^2$, and observing that \eqref{ephi} implies $\mathscr{L}(w)\ge 0$. 
{}From the latter, it follows by the maximum principle that $|\varphi_h(x,t)|\le 1$ 
for all $(x,t)\in \Om\times [0,t_0]$. 

Now, set 
\begin{eqnarray*}
q(x,t):=\,e^{\frac{1}{M}\big(t_0-t-|x|\big)},
\end{eqnarray*}
with $M>\sup_{\Om\X I}f_u(x,u)$, $I\supset[-\|u_\s\|_\infty,\|u_\s\|_\infty]$ for $0<\s<1$. We will use the maximum principle to prove that $w\le q$. This is obviously true inside the
cone $|x|\le t_0-t$, where $q\ge1$.   We also note that 
$$
\mathscr{L}(q)=-q\bigg{\{}\frac{1}{M}\bigg(1-\frac{f^{\s}_u(x,u)}{M}\bigg)
+\frac{f^{\s}_u(x,u)}{M}\frac{n-1}{|x|}\bigg{\}}\le 0,
$$
which yields $\mathscr{L}(w-q)\ge 0$. It is easily seen that 
$$
w-q|_{\po\Om\X[0,t_0]}=-q|_{\po\Om\X[0,t_0]}\le 0,\qquad
w(x,t_0)-q(x,t_0)\le 0.
$$
Then, the claim follows by the maximum principle (cf., e.g., \cite{PW}).

3. Let $0\le \rho \in C_c^{\infty}(\re)$ with $\supp\,\rho\subset[-1,1]$ and $\int_{\R}\rho\,dx=1$. Set
$$
\eta_m(\lambda):=1-\int_{-\infty}^{\lambda}\rho(s-m)\,ds,
$$
for $m\in\mathbb{N}$, and take 
$$
g^k(x,t):=\varphi^k_h(x,t)\,\eta_m(|x|)
$$
as a test function in~\eqref{kr2}. Hence
\begin{align}\label{kr5}
&\int_{\Om}\sum_{k=1}^nv^k(t_0)\bigl(\sgn (v^k(t_0))\chi_{{}_{\Om_h}}\bigr)\ast \rho_h\, e^{-|x|}
\eta_m(|x|)\,dx=
\sum_{k=1}^n\int_0^{t_0}\int_{\Om}\bigg{\{}2f^{\s}(x,u)\nabla \varphi^k_h\cdot \nabla \eta_m(|x|)\nonumber\\
&\qquad\qquad
+f^{\s}(x,u)\varphi^k_h\Delta\eta_m(|x|)-\sum_{i=1}^n
f^{\s}_{x_iu}(x,u)\partial_{x_k}\eta_m(|x|)\varphi^k_h\bigg{\}}\,dx\,dt
\nonumber\\&\qquad\qquad+
\int_{\Om}\sum_{k=1}^n\bigg{\{}v^k(0)\varphi^k_h(x,0)
-(\Delta f)(x,u)\varphi^k_h(x,t_0)\bigg{\}}\eta_m(|x|)\,dx.
\end{align}

Thus, letting $m\to\infty$ first and then letting $h\to 0$, we obtain an 
estimate of the form
\begin{eqnarray*}
\int_{\Om}\sum_{k=1}^n|v^k(t_0)|\,e^{-|x|}\,dx\le
c(T,\|\nabla u_0{\|}_{\infty}, \sum_{i=1}^n \|f_{x_ix_i}\|_\infty)<\infty,
\end{eqnarray*}
for all $t_0\in [0,T]$, where, in particular, the right-hand side does not depend on $\s$. Consequently,  we get
$$
\int_{\Om}|u_{\s}(x+y,t)-u_{\s}(x,t)|\xi_R(x)\,dx\le
c_1|y|,
$$
for some $c_1$ independent of $\s$, which gives~\eqref{e05}.

4. To prove~\eqref{e06}, we first note that from \eqref{eTcomp1} and the hypotheses on $f$, we know that there exists $M>0$ such that 
$|f^{\s}(x,u_{\s}(x,t))|\le M$ for all $(x,t)\in\Om\X[0,\infty)$  and 
for all $\s>0$. Now, fix $t,s,\s$ and set 
$w(x):= u_{\s}(x,t+s)-u_{\s}(x,t)$. Given $\phi \in
W^{2,\infty}(\Om)$, we obtain

\begin{align*}
\int_{\Om}w(x)\phi(x)\xi_R(x)\,dx&=\int_{\Om}\int_t^{t+s}
{\partial}_tu_{\s}(x,\tau)\phi \xi_R \,d\tau \,dx=\int_{\Om}\int_t^{t+s}
\Delta f^{\s}(x,u_{\s})\phi \xi_R
 \,d\tau \,dx\\
&=\int_{\Om}\int_t^{t+s}f^{\s}(x,u_{\s})\Delta(\phi \xi_R)\,d\tau\,dx
\\
&=\int_{\Om}\int_t^{t+s}\bigg{\{} f^{\s}(x,u_{\s})\Delta \phi \xi_R 
+2f^{\s}(x,u_{\s})\nabla \phi \cdot 
\nabla \xi_R +f^{\s}(x,u_{\s})\phi \Delta \xi_R 
\bigg{\}} \,d\tau\,dx,
\end{align*}
and this implies

\begin{align}\label{e07}
\bigg| \int_{\Om}w(x)\phi(x)\xi_R(x)\,dx\bigg|&\le
M\bigg{\{} \|\Delta \phi {\|}_{\infty}+2\|\nabla \phi 
{\|}_{\infty}
+\|\phi {\|}_{\infty} \bigg{\}}\|\xi_R {\|}_{1}s.
\end{align}

Taking $\phi = (\sgn\,w)\ast {\rho}_{\delta}$, with $\sgn w$ extended by 0 outside $\Om$,  and observing that 
$\|\nabla \phi{\|}_{\infty} \le \frac{c}{\delta}, 
\|\Delta \phi {\|}_{\infty}\le 
\frac{c}{\delta^2}$ and $\|\phi{\|}_{\infty}\le 1$, where $c$ only 
depends on the dimension, we get

\begin{align*}
\int_{\Om}|w(x)|\xi_R(x)\,dx&=\bigg(\int_{\Om}w(x)\,\sgn(w(x))\,
\xi_R(x)\,dx\bigg)\int_{{\re}^n}\rho(y)\,dy\\
&=\int_{\Om\times \R^n}w(x-\delta y)\,\sgn(w(x-\delta y))\,
\xi_R(x-\delta y)\rho(y)\,dx\,dy,
\end{align*}

and

\begin{align*}
\int_{\Om}w(x)\varphi(x)\xi_R(x)\,dx&=\int_{\Om}w(x)\xi_R(x)
\bigg(\int_{\Om}\sgn(w(y))\,{\rho}_{\delta}(x-y)\,dy\bigg)\,dx\\
&=\int_{\Om}w(x)\xi_R(x)\bigg(\int_{{\re}^n}\sgn(w(x-\delta y))\rho(y)
\,dy\bigg)\,dx\\
&=\int_{\Om\times \R^n}w(x)\,\xi_R(x)\,\sgn(w(x-\delta y))\,\rho(y)\,dx\,dy.
\end{align*}

Hence,
\begin{align*}
&\int\limits_{\Om}|w(x)|\xi_R(x)\,dx-\int_{\Om}w(x)\phi(x)\xi_R(x)
\,dx\\
&=\int\limits_{\Om\times\R^n}\bigg{\{}w(x-\delta y)\,\sgn(w(x-\delta y))\,
\xi_R(x-\delta y)-w(x)\,\xi_R(x)\,\sgn(w(x-\delta y))\bigg{\}}\rho(y)
\,dx\,dy\\
&=\int\limits_{\Om\times \R^n}\bigg{\{}\big[w(x-\delta y)-w(x)\big]\,
\sgn(w(x-\delta y))\xi_R(x)+\big[\xi_R(x-\delta y)-\xi_R(x)\big]\,
\sgn(w(x-\delta y))\,w(x-\delta y)\bigg{\}}\rho(y)\,dx\,dy.
\end{align*}

Therefore,
\begin{equation}\label{e08}
\bigg|\int_{\Om}|w(x)|\xi_R(x)\,dx-\int_{\Om}w(x)\phi(x) \xi_R(x)\,dx
\bigg|\le (2c_0+\|\nabla\xi_R\|_1)\delta.
\end{equation}

Thus, we conclude from~\eqref{e08} and from~\eqref{e07} that

\begin{equation*}
\int_{\Om}|w(x)|\xi_R(x)\,dx\le\,(2c_0+\|\nabla\xi_R\|_1)\delta+ s\,M \bigg{\{}
\frac{1}{\delta^2}+\frac{2}{\delta}+1\bigg{\}}\|\xi_R{\|}_1,
\end{equation*}
for all $0<\delta<1$, which completes the proof.

\end{proof}

\begin{theorem}[Existence]\label{t01}
Let $u_\s$ be the unique solution of \eqref{e04},\eqref{e04'},\eqref{e04''}, and $u_0\in W^{1,\infty}(\Om)$. There exists $u\in L^{\infty}(\Om\X[0,\infty))$ such that, passing to a 
suitable subsequence if necessary, 
$u_{\s}\to u$ a.e.\ in $\Om\X[0,\infty)$ as $\s \to 0$.
Moreover, $u$ is the unique entropy solution of~\eqref{e01}-\eqref{e01''}. Consequently, using the stability in $L_\loc^1(\Om)$ of entropy solutions, we obtain the existence of a unique entropy solution also for $u_0\in L^\infty(\Om)$.
\end{theorem}
\begin{proof}

1. We first treat the case where $f(x,u)$ is of type~1 and $u_0\in W^{1,\infty}(\Om)$.  Let $g^\s(x,v)$ be such that 
\begin{equation}\label{et01.0}
g^\s(x,f^\s(x,u))=u, \quad   f^\s(x,g^\s(x,v))=v. 
\end{equation}
We claim that  $g^\s(x,v)$ converges locally uniformly in $\Om\X\R$ to $g(x,v)$ satisfying $g(x,f(x,u))=u$ and $f(x,g(x,v))=v$. 

Indeed, by construction $f^\s(x,u)$ clearly converges locally uniformly to $f(x,u)$. Now, given any compact $K\subset\Om$ and a bounded interval $I\subset\R$, $g^\s(x,v)$ is uniformly bounded on $K\X I$, by  \eqref{et01.0}, and, so, $g^\s(x,v)\in J$, for some bounded interval $J$, for $(x,v)\in K\X I$. Now, for $\s$ sufficiently close to $0$, $f^\s(x,u)$ is arbitrarily close to $f(x,u)$, uniformly for $(x,u)\in K\X J$. Therefore, given $\ve'>0$, there exists $\s_0>0$, such that, if $0<\s<\s_0$, $|f^\s(x,g^\s(x,v))-f(x,g^\s(x,v))|<\ve'$, that is
$|f(x,g^\s(x,v))-v|<\ve'$, which implies that $g(x,v-\ve')<g^\s(x,v)<g(x,v+\ve')$ for all  $(x,v)\in K\X I$. Hence, given $\ve>0$, we choose $\ve'>0$ such that $g(x,v+\ve')<g(x,v)+\ve$
and $g(x,v-\ve')>g(x,v)-\ve$, for all $(x,v)\in K\X I$, to get $|g^\s(x,v)-g(x,v)|<\ve$, for $0<\s<\s_0$, for all $(x,v)\in K\X I$, proving the assertion.   

2. Now, let $v_\s(x,t)=f^\s(x,u_\s(x,t))$. We have
\begin{equation}\label{et01.1}
 g_v^\s (x,v_\s(x,t))\po_t v_\s- \D v_\s =0, \quad (x,t)\in\Om\X(0,\infty).
 \end{equation}
 We multiply \eqref{et01.1} by $e^{-|x|} \po_t v_\s$, integrate over $\Om$, and we use the fact that $g_v^\s(x,v)>\d_0>0$, for some $\d_0>0$ independent of $\s$, for $(x,v)\in\Om\X[-M,M]$, with $M>0$ sufficiently large so that $\|v_\s\|_\infty<M$,  to obtain 
 \begin{equation} \label{et01.2}
 \frac{\d_0}2\int_{\Om}(\po_t v_\s(x,t))^2 e^{-|x|}\,dx+\frac12\frac{d}{dt}\int_{\Om} |\nabla v_\s(x,t)|^2 e^{-|x|}\,dt  \le  C(\d_0)\int_{\Om} |\nabla v_\s|^2 e^{-|x|}\,dx.
 \end{equation}
 By Gronwall inequality we then obtain 
 \begin{equation}\label{et01.3}
 \int_0^T\int_{\Om}\bigl((\po_t v_\s(x,t))^2+ |\nabla v_\s(x,t)|^2\bigr) e^{-|x|}\,dx\,dt\le C(T),
 \end{equation}
 for $0\le t\le T$, for some $C(T)>0$ independent of $\s$, for all $T>0$.  Inequality \eqref{et01.3}, indicates that $v_\s$ is uniformly bounded in $W_\loc^{1,2}(\bar\Om\X(0,\infty))$. Therefore, by  the well known Sobolev embedding, we may extract a subsequence of $v_\s(x,t)$,  still denoted by $v_\s(x,t)$, converging, in  $L_\loc^1(\Om\X(0,\infty))$, to some 
 $v(x,t)\in  W_\loc^{1,2}(\bar\Om\X(0,\infty)$. Since $u_\s(x,t)=g^\s(x,v_\s(x,t))$, we have that $u_\s(x,t)$ converges to $u(x,t)=g(x,v(x,t))$ in $L_\loc^1(\Om\X(0,\infty))$.

3. Now, we assume $f(x,u)$ is of type~2 and $u_0\in W^{1,\infty}(\Om)$.
By Theorem~\ref{l04}, for each $t>0$,
$\{u_{\s}(t) {\}}_{\s>0}$ is a sequence uniformly bounded in $BV_\loc(\Om)$, and it is an equicontinuous family in $C([0,T];L^1_\loc(\Om))$. Therefore, by the well known compactness of the embedding $BV_\loc(\Om)\subset L_\loc^1(\Om)$ (see, e.g., \cite{EG}), there exists 
$u\in L^{\infty}(\Om\X(0,\infty))$ such that, passing to a subsequence if necessary, 
$u_{\s} \to u$ in $L^1_{\loc}(\Om\X(0,\infty))$. 

4. In this step we prove that, for  any $R>0$,   $f^\s (x,u_\s(x,t))$ is uniformly bounded in $L^2([0,T]; H_0^1(\Om\cap B(0;R)))$ by a constant $C(R,T,\|u_0\|_\infty)$, depending only on $R,T,\|u_0\|_\infty$, in particular, not depending on $\|\nabla u_0\|_\infty$. 

For this, we  multiply \eqref{e04} by $f^\s(x,u_\s(x,t)) e^{-|x|}$ and integrate in $\Om\X(0,T)$, using iteration by parts to get 
$$
\int_0^T\int_{\Om}\bigg{\{}{\partial}_t u_{\s} f^{\s}(x,u_{\s})
e^{-|x|}+\nabla f^{\s}(x,u_{\s})
\cdot \nabla(f^{\s}(x,u_{\s})e^{-|x|})\bigg{\}} \,dx\,dt
=0,
$$
which yields  
\begin{align*}
&\int_0^T\int_{\Om}e^{-|x|}\, {\partial}_t\bigg[\int_0^{u_{\s}}
f^{\s}(x,s)ds \bigg]\,dx\,dt+\int_0^T\int_{\Om}
|\nabla f^{\s}(x,u_{\s})|^2 e^{-|x|}
+f^{\s}(x,u_{\s})\nabla f^{\s}(x,u_{\s})\cdot \nabla e^{-|x|} \bigg{\}}\,dx\,dt=0,
\end{align*}
and so 
\begin{align*}
&\int_0^T\int_{\Om}|\nabla f^{\s}(x,u_{\s})|^2 e^{-|x|} \,dx\,dt \le C\left\{\int_0^T\int_{\Om} |f^{\s}(x,u_{\s})|^2 e^{-|x|}\,dx\,dt +\int_{\Om}e^{-|x|}\bigg|\int_{u_{0}}^{u_{\s}(T)}
f^{\s}(x,s)\,ds\bigg|\,dx\right\}.
\end{align*}
Therefore,
$$
\int_0^T\int_{\Om\cap B(0;R)}|\nabla f^{\s}(x,u_{\s})|^2
\,dx\,dt \le
c(T,R, \|u_0{\|}_{\infty})
$$
for all $0<\s<1$, as claimed. 
In particular, $f^\s(x, u_\s(x,t))$ is uniformly bounded in $L^2_\loc((0,\infty); H_{0,\loc}^1(\bar\Om))$, and so $f(x,u(x,t))\in L^2_\loc((0,\infty); H_{0,\loc}^1(\bar\Om))$. That $u(x,t)$ is an entropy solution of \eqref{e01}-\eqref{e01''} follows from the latter and the convergence of $u_\s(x,t)$ in $L^1_\loc(\Om\X(0,\infty))$. 

5. Finally, when $u_0\in L^{\infty}(\Om)$, we may approximate $u_0$ in $L_\loc^1(\Om)$ by a sequence $u_{0k}\in W^{1,\infty}(\Om)$ obtaining a sequence $u_k$ of entropy solutions of \eqref{e01}-\eqref{e01''}, with initial data $u_0=u_{0k}$, and then use the stability Theorem~\ref{T:A4} to deduce that $u_k$ is a Cauchy sequence in 
$L_\loc^1(\Om\X(0,\infty))$.
We then easily conclude that the limit $u\in L^\infty(\Om\X(0,\infty))$ of the sequence $u_k$ is an entropy solution of \eqref{e01}-\eqref{e01''}.

\end{proof}

We close this section by establishing an elementary result which will be needed in the following sections.

\begin{lemma}\label{L:g} Let $f(x,u)$ be either of type 1 or type 2, and let $g(x,v)$ be the function  left-continuous on $v$  determined by the relation $f(x,g(x,v))=v$, for all $x\in\Om,\ v\in\R$. Then,   
$\lim_{v\to\pm\infty}g(x,v)=\pm\infty$, uniformly in $x$.
\end{lemma}

\begin{proof} Let us prove that $\lim_{v\to+\infty}g(x,v)=+\infty$ uniformly in $x$. From {\bf(f1.1)}, when $f$ of type 1,  or {\bf(f2.1)}, when $f$ is of type 2, there exist $u_1<0< u_2$ such that $f(x,u_1)\le0\le f(x,u_2)$, for all $x$. 
Now, given any $M>0$,  if $M':=\max\{|u_1|,u_2,M\}$,  then, for $M'\le u\le 2M'$, we have $0\le f(x,u)\le f(x,u)-f(x,u_2) \le 3CM'$,
where $C>0$ is the uniform in $x$ Lipschitz constant of $f(x,\cdot)$ on $[-M',2M']$. Hence, $g(x,3CM')\ge u\ge M$, for all $x$, and, since $g(x,\cdot)$ is increasing, we have
$g(x,v)> M$ for all $v>3CM'$, uniformly in $x$. This concludes the proof that    $\lim_{v\to+\infty}g(x,v)=+\infty$ uniformly in $x$; the proof that $\lim_{v\to-\infty}g(x,v)=-\infty$, uniformly in $x$, is completely similar.
\end{proof}

\section{Homogenization of Porous Medium Type Equations: \\ Unbounded domains, general ergodic algebras and well-prepared initial data}\label{S:5}

In this and the next sections, we consider the following homogenization problem 

\begin{equation} \label{p1}
\begin{cases}
{\partial}_tu=\Delta f(x,\frac{x}{\ve}, u), &(x,t)\in \Om\X(0,\infty),\\ 
u(x,0)=u_{0}(x,\frac{x}{\ve}), &x\in \Om,\\
 f(x,\frac{x}{\ve},u(x,t))=0, &(x,t)\in\po\Om\X(0,\infty),
\end{cases}
\end{equation}
where $f:\Om\X\R^n\X\R\to\R$ is  a continuous function such that, for each $(x,z)\in\Om\X\R^n$, $f(x,z,\cdot)$ is locally Lipschitz, uniformly with respect to $(x,z)$, and, for each 
$(x,u)\in\Om\X\R$, $f(x,\cdot,u)\in\AA(\R^n)$, where $\AA(\R^n)$ is some given ergodic algebra. Here, as in the previous section, $\Om\subset\R^n$ is an open set, possibly unbounded, with  smooth boundary. 

In this section we will be concerned with the case where $\AA(\R^n)$ may be a general ergodic algebra, but we will need to restrict our initial data to the class of well-prepared ones, which we will define subsequently.  Following the discussion in the previous section, we consider two different situations, according to whether, for all $\ve>0$,   $f_\ve(x,u):= f(x,\frac{x}{\ve},u)$ is of type~1 or of type~2, as defined in the previous section.

For the first situation, we have the following  assumption. 
\begin{enumerate}
\item[{\bf(h1.1)}]  In the case where $f_\ve(x,u)$ is of type~1, for all $\ve>0$, we assume that $f(\cdot,z,\cdot)$ satisfies {\bf(f1.1)} uniformly with respect to $z\in\R^n$, and 
$f(x,z,0)=0$, for all $(x,z)\in\po\Om\X\R^n$. Also, let $g(x,z,v)$ be such that $g(x,z,f(x,z,u))=u$ and $f(x,z,g(x,z,v))=v$, so that $g:\bar\Om\X\R^n\X\R\to\R$ is a continuous function. We assume that, for each $(x,v)\in\bar \Om\X\R$, $g(x,\cdot,v)\in\AA(\R^n)$.  
\end{enumerate}

We  define the function $\bar g:\bar \Om\X\re\to\re$ by
\begin{equation}\label{egbar}
\bar g(x,v)=\int_{\KK} g(x,z,v)\,d\mm(z),
\end{equation}
where $\KK,\mm$ are the compact space and the invariant measure associated with the ergodic algebra $\AA(\R^n)$, and  $\bar f:\bar\Om\X\to\R$ by
\begin{equation}\label{efbar}
\bar g(x,\bar f(x,u))=u.
\end{equation}

We have the following lemma.

\begin{lemma}\label{L:6.1} If  $f_\ve(x,u)$ is of type~1 for all $\ve>0$ and {\bf(h1.1)} holds, then $\bar f:\bar \Om\X\R\to\R$, defined by \eqref{efbar}, is of type~1.
\end{lemma}

\begin{proof}  We first observe that the function $\bar g:\bar\Om\X\R\to\R$, defined in \eqref{egbar} is continuous on $\bar\Om\X\R$, which follows directly from the continuity of $g$.  
Clearly, from the fact that $g$ is strictly increasing in $v$, it follows that $\bar g$ is strictly increasing in $v$, and so, $\bar f$ is well defined by \eqref{efbar}, and, from the continuity of $\bar g$, we deduce the continuity of 
$\bar f$ on $\bar\Om\X\R$.  Indeed, since {\bf(f1.1)} holds for $f(\cdot,z,\cdot)$, uniformly in $z\in\R^n$, the assertion of Lemma~\ref{L:g} also holds for $g(\cdot,z,\cdot)$ uniformly in 
$z\in\R^n$, and, hence, also for $\bar g(x,v)$.  In particular, $\bar g(x,v)$ remains bounded if, and only if, $v$ remains confined on a bounded subset of $\R$, uniformly for $x\in\bar\Om$. Thence, from \eqref{efbar}, if $((x_k,u_k))_{k\in\N}$ is a sequence in $\bar \Om\X\R$, converging to $(x_*,u_*)$, then $(\bar f(x_k,u_k))_{k\in\N}$ remains bounded, and, passing to any converging subsequence, still denoted $(\bar f(x_k,u_k))$, we conclude that $\bar f(x_k,u_k)$ must converge to $\bar f(x_*,u_*)$, which means that the whole sequence $\bar f(x_k,u_k)$ converges to $\bar f(x_*,u_*)$, proving the continuity of $\bar f$ in $\bar\Om\X\R$. 
We also see that 
$$
\lim_{u\to\pm\infty} \bar f(x,u)=\pm\infty,
$$
uniformly with respect to $x\in\bar\Om$, since this is true for $\bar g$, as we have just seen,  and we may apply the same reasoning as in the proof of Lemma~\ref{L:g}.  

As for the local Lipschitz continuity of $\bar f(x,\cdot)$, uniformly with respect to $x\in\bar\Om$, we have the following. Given any $M>0$, let $K>0$ be such that $f(x,z,-K)\le \bar f(x,-M)$ and $f(x,z,K)\ge \bar f(x,M)$, for all $(x,z)\in\bar\Om\X\KK$, and let $C>0$ be the uniform Lipschitz constant of $f(x,z,\cdot)$ on the interval $[-K,K]$, for all $(x,z)\in\bar\Om\X\KK$. We have
$$
|g(x,z,v_1)-g(x,z,v_2)|\ge C^{-1}|v_1-v_2|, \qquad \text{for $v_1,v_2\in [\bar f(x,-M),\bar f(x,M)]$}.
$$  
Therefore, for $u_1,u_2\in[-M,M]$, using the monotonicity of $g(x,z,\cdot)$, we get
\begin{equation}\label{eLip}
|u_1-u_2|= \int_{\KK} |g(x,z,\bar f(x,u_1))-g(x,z,\bar f(x,u_2))|\,d\mm(z) \ge C^{-1}|\bar f(x,u_1)-\bar f(x,u_2)|,
\end{equation}
for all $x\in\bar\Om$, which concludes the proof.
\end{proof}

Let us now analyze the case where $f_\ve$ is of type~2, for all $\ve>0$.
 
\begin{enumerate}
\item[{\bf(h2.1)}] In the case where $f_\ve$ is of type 2, for all $\ve>0$,  we assume that  $f(x,z,u)=h(x,z)F(u)+S(x,z)$, with $F$ satisfying {\bf(f2.1)}, $h_\ve, S_\ve$ satisfying {\bf(f2.2)}, {\bf(f2.3)}, for each $\ve>0$, where $h_\ve(x):=h(x,\frac{x}{\ve})$,  $S_\ve(x):=S(x,\frac{x}{\ve})$,  and we assume further that $h,S: \bar\Om\X\R^n\to \R$ are bounded continuous functions, with  $h(x,z)>\d_0>0$ for all $(x,z)\in\bar\Om\X\R^n$, and $S(x,z)=0$ for all $(x,z)\in\po\Om\X\R^n$.  Moreover, for each $x\in\bar\Om$, $h(x,\cdot),\, S(x,\cdot)$  belong to a given ergodic algebra  $\AA(\re^n)$. \\
 We define $g(x,z,v):=G\left(\frac{v-S(x,z)}{h(x,z)}\right)$, where $G:\R\to\R$ is the right-continuous function 
satisfying $F(G(v))=v$,   as in Definition~\ref{D2'}, and we let $E$ be the (countable) set of discontinuities of $G$.\\
 We assume further that, for all $\a\in E$, setting   $\psi_\a(x,z):=\a h(x,z)+S(x,z)$,  we have 
\begin{equation}\label{em0}
\mm\left(\{z\in\KK\,:\, \psi_\a(x,z)=v\}\right)=0, \quad\text{for all $(x,v)\in\bar\Om\X\R$}.
\end{equation}
\end{enumerate}

We remark that Lemma~\ref{L:zeromeasure} gives sufficient conditions in order for $\psi_\a$ to satisfy \eqref{em0}.  More specifically, the item (ii) in the statement of 
Lemma~\ref{L:zeromeasure}, for example, asserts that the condition is satisfied if, for each $x\in\bar\Om$,  $\psi_\a(x,\cdot), \nabla_z\psi_\a(x,\cdot),\nabla_z^2\psi_\a(x,\cdot)\in\AA(\R^n)$ and $|\nabla_z\psi_\a(x,z)|^2+|\nabla_z^2\psi_\a(x,z)|^2>\d_x>0$, for all $z\in\R^n$. We recall also that $E$ is countable and may be viewed also as a bounded set, since we will be dealing with sequences of functions assuming values in a fixed bounded interval of $\R$. 

Assumption \eqref{em0},  in {\bf(h2.1)}, makes it possible  to define $\bar g$ again by \eqref{egbar}, and we have the following.

\begin{lemma}\label{L:6.2}  In the case where $f_\ve$ is of type~2, for all $\ve>0$, and {\bf(h2.1)} holds, $\bar g(x,\cdot):\R\to\R$ is strictly increasing and continuous, for any $x\in\bar\Om$, with  $\lim_{v\to\pm\infty} \bar g(x,v)=\pm\infty$, uniformly with respect to $x\in\bar \Om$.   Moreover, $\bar f:\bar\Om\X\R\to\R$, defined by \eqref{efbar}, is of type~1.
\end{lemma}
 
 \begin{proof}  The fact that $\bar g(x,\cdot)$ is strictly increasing follows directly from the definition, since $g(x,z,\cdot)$ is strictly increasing, while the fact that it is continuous follows from an easy application of the dominated convergence theorem, as a consequence of  \eqref{em0}. The proofs of the facts that $\bar f(x,\cdot)$ is locally Lipschitz continuous, uniformly with respect to $x\in\bar\Om$, and that $\lim_{u\to\pm\infty}\bar f(x,u)=\pm\infty$, uniformly with respect to $x\in\bar\Om$, are similar to the proofs of the corresponding assertions for the case where $f_\ve(x,u)$ is of type~1, for all $\ve>0$.
 \end{proof}

Concerning  the initial data, in the case where $f_\ve$ is of type~1, for all $\ve>0$, we assume
\begin{enumerate}
\item[{\bf(h1.2)}]  $u_0(z,x)=g(z, \phi_0(x))$ with $\phi_0\in L^\infty(\Om)$,
\end{enumerate}
while, in the case where $f_\ve$ is of type~2, for all $\ve>0$, we assume
\begin{enumerate}
\item[{\bf(h2.2)}] $u_0(z,x)=G\left((\phi_0(x)-S(x,z))/h(x,z)\right)$, with $\phi_0\in L^\infty(\Om)$. 
\end{enumerate}

The particular form of the initial data prescribed in {\bf(h1.2)} and {\bf(h2.2)} is sometimes summarized by saying that the initial data are {\em well-prepared}.

We define
\begin{equation}\label{bar1}
\bar u_0(x)=\Medint u_0(x,z)\,dz.
\end{equation}
Observe that, by the hypotheses {\bf(h1.2)}, when $f$ is of type~1, or {\bf(h2.2)}, when $f$ is of type~2,  \eqref{bar1} is equivalent to $\bar u_0(x)=\bar g(\phi_0(x))$. 

For each $\a\in\re$, we define 
\begin{equation}\label{estead}
\Phi_{\alpha}(x,z):=  g(x,z,\alpha), \quad (x,z)\in\bar\Om\X\R^n. 
\end{equation}

In the proof of Theorem~\ref{T:6.1} below,  we will use the fact that  $\Phi_{\a}(x,\frac{x}{\ve})$ trivially  satisfies \eqref{e02}, with $f$ replaced by $f_\ve$, if $f_\ve$ is of type~1, 
for all $\ve>0$,  or \eqref{eD2''}, with $f$ replaced by $f_\ve$, if $f_\ve$ is of type~2, for all $\ve>0$,  and, obviously,  $\a= f_\ve(x,\Phi_\a(x,\frac{x}{\ve}))\in L_\loc^2((0,\infty);H_\loc^1(\Om))$. Therefore, it satisfies the assumptions in Remark~\ref{obs1}.  In particular, given any entropy solution of \eqref{p1}, $u_\ve(x,t)$, we may apply Corollary~\ref{cor1},
for $u_1(x,t)=u_\ve(x,t)$ and $u_2(x,t)=\Phi_\a(x,\frac{x}{\ve})$, as long as \eqref{ecor1} is verified.

\begin{theorem}\label{T:6.1}
Let $u_\ve(x,t)$ be the entropy solution of~\eqref{p1}.  For $f_\ve$ of type~1,  assume that  {\bf(h1.1)}, {\bf(h1.2)} hold;  for $f_\ve$  of type~2, assume that  {\bf(h2.1)} and {\bf(h2.2)} hold.  
Then $u_\ve$  weak star converge in $L^\infty(Q)$ to $\bar u(x,t)$, where the latter is the entropy solution to the problem
\begin{equation}\label{p4}
\begin{cases}
{\partial}_t\bar{u}=\Delta \bar{f}(x,\bar{u}),
& (x,t)\in Q=\Om\X(0,\infty), \\ 
\bar{u}(x,0)=\bar u_0(x), 
& x\in \Om,\\
\bar u(x,t) =0, &(x,t)\in\po\Om\X(0,\infty).
\end{cases}
\end{equation} 
Moreover,  we have
\begin{equation}\label{p6}
u_\ve(x,t)- g\bigl(x,\frac{x}{\ve},\bar f(x, \bar u(x,t))\bigr)\to 0,
 \qquad\text{as $\ve\to0$ in $L_\loc^1(Q)$}.
\end{equation}
\end{theorem}

\begin{proof}
Since the proof  is similar to that of theorem~7.1 in \cite{AFS} with slight modifications,  we will only outline the main steps of it, although we believe the sketch presented here will be enough for a complete understanding, in a self-contained way.     In any case,  for  details omitted here, we refer the reader to~\cite{AFS}. The case where $f_\ve$ is of type~2 will demand some specific considerations which are fully provided here.

1. First, we  observe that the weak solutions $u_{\ve}$, $\ve>0$, of \eqref{p1} are bounded uniformly with respect to $\ve$ in $L^{\infty}(Q)$. Indeed, we note that 
if $\a_1,\a_2$ are such that $\a_1\le \phi_0(x)\le \a_2$ for $x\in\R$, we have 
 \begin{equation}\label{eT61.1}
 \Phi_{\a_1}(x,\frac{x}{\ve})\le u_{0}(x,\frac{x}{\ve}) \le \Phi_{\a_2}(x,\frac{x}{\ve}) \qquad \text{for all $x\in\Om$}.
 \end{equation}
 So, choosing adequately $\a_1,\a_2\in \R$ in \eqref{eT61.1}, by the comments made just before the statement,  we may apply Corollary~\ref{cor1} to get
$$
\Phi_{\a_1}(x,\frac{x}{\ve}) \le u_{\ve}(x,t)\le \Phi_{\a_2}(x,\frac{x}{\ve})\qquad \text{for a.e.\ $(x,t)\in Q$}.
$$
Choosing $A_1,A_2\in \R$ such that $f(x,z,A_1)\le \a_1$ and $f(x,z,A_2)\ge \a_2$ for all $(x,z)\in\Om\X\R^n$,  since then $A_1\le \Phi_{\a_1}(x,\frac{x}{\ve})$ and 
$A_2\ge \Phi_{\a_2}(x,\frac{x}{\ve})$, for all $x\in\Om$, we obtain  a compact interval $K=[A_1,A_2]$ in which  $u_\ve(x,t)$ assumes its values for a.e.\ $(x,t)\in\Om\X(0,\infty)$.

Let ${\nu}_{z,x,t}\in \mathcal{M}(K)$, with $(z,x,t)\in\KK\X Q$, be the two-scale space-time Young measures associated with a subnet of $\{u_{\ve}{\}}_{\ve >0}$ with test functions oscillating only on the space variable. Following \cite{ES}, \cite{AF} and \cite{AFS}, the theorem is proved by adapting DiPerna's method in \cite{DP}, that is, by showing that ${\nu}_{z,x,t}$ is a Dirac measure for almost all $(z,x,t)\in {\mathcal K} \X Q$. Since we are going to show that ${\nu}_{z,x,t}$ does not 
depend on the chosen subnet (so that, a posteriori, a full limit as 
$\ve\to 0$ occurs), in order to simplify our notation we will use the 
notation $\lim_{\ve\to 0}$, with no reference to the subnet.

Observe that, for every $\a\in\R$, the entropy solutions $u_{\ve}$
and ${\Phi}_{\alpha}(\frac{x}{\ve}):=g(\frac{x}{\ve},\a)$ satisfy (see Theorem~\ref{T:A3})
\begin{multline}\label{p8}
\int_{Q}|u_{\ve}(x,t)-{\Phi}_{\alpha}(x,\frac{x}{\ve})|
\phi_t+ |f(x,\frac{x}{\ve},u_{\ve}(x,t))-f(x,\frac{x}{\ve},{\Phi}_{\alpha}
(x,\frac{x}{\ve}))|\Delta \phi \,dx\,dt\\
+ \int_{\Om} |u_0(x,\frac{x}{\ve})-{\Phi}_{\alpha}(x,\frac{x}{\ve})|\phi(x,0)\,dx\ge 0,
\end{multline}
for all $0\le \phi\in C_c^\infty(\Om\X\R)$. 
In \eqref{p8}, we take $\phi(x,t)={\ve}^2 \varphi(\frac{x}{\ve})\psi(x,t)$ with 
$0\le \psi \in C_c^{\infty}({\re}_{+}^{n+1})$, 
$\varphi,\, \Delta \varphi \in \mathcal{A}(\re^n)$ and $\varphi \ge 0$. Observe that 
$$
\Delta \phi = \Delta \varphi(\frac{x}{\ve})\psi(x,t)
+2\ve \nabla \varphi(\frac{x}{\ve})\cdot \nabla \psi(x,t)
+{\ve}^2\varphi(\frac{x}{\ve})\Delta \psi(x,t).
$$
Letting $\ve \to 0$ and using Theorem~\ref{T3}, we get
$$
\int_{Q}\int_{{\mathcal K}}\psi(x,t)\langle {\nu}_{z,x,t},
|f(x,z,\cdot)-f(x,z,{\Phi}_{\alpha}(x,z))|\rangle \underline{\Delta \varphi}(z) 
\,d{{\mathfrak m}}(z)\,dx\,dt \ge 0.
$$
Now apply the inequality above to $\| \varphi {\|}_{\infty}\pm \varphi$ to
obtain
\begin{equation}\label{p10}
\int_{Q}\int_{{\mathcal K}}\psi(x,t)\langle {\nu}_{z,x,t},
|f(x,z,\cdot)-\alpha|\rangle \underline{\Delta \varphi}(z)\,
d{\mathfrak m}(z)\,dx\,dt=0 ,
\end{equation}
for all $\varphi$ such that $\varphi, \Delta \varphi \in \mathcal{A}(\re^n)$ 
and all $0\le \psi\in C^{\infty}_c(Q)$.  Moreover, equality \eqref{p10} also holds if we replace $|f(x,z,\cdot)-\a|$ by $f(x,z,\cdot)-\a$, which is achieved in the same way by using the integral equality in the definition of weak solution instead of the entropy inequality. 
Therefore, we obtain 
\begin{equation}\label{p10'}
\int_{Q}\int_{{\mathcal K}}\psi(x,t)\langle {\nu}_{z,x,t},
\theta(f(x,z,\cdot))\rangle \underline{\Delta \varphi}(z)\,
d{\mathfrak m}(z)\,dx\,dt=0,
\end{equation}
for any affine function $\theta$, and, by approximation, we get that \eqref{p10'} holds for any $\theta\in C(K')$, where $K' $ is a compact interval such that $f(x,z,K)\subset K'$, for all 
$(x,z)\in \bar\Om\X\R^n$.

2. Define a new family of parametrized measures ${\mu}_{z,x,t}$ given by 

\begin{equation}\label{p11}
\langle {\mu}_{z,x,t},\theta \rangle := \langle {\nu}_{z,x,t},
\theta (f(x,z,\cdot))\rangle, \qquad \theta \in C(K').
\end{equation}
By~\eqref{p10'}, we have 
\begin{equation}\label{ep10'}
\Delta_z\langle\mu_{z,x,t},\theta\rangle =0,\qquad
\text{in the sense of $\BB^2$.}
\end{equation}
Therefore, by the ergodicity of $\AA(\R^n)$, using Lemma~\ref{L:ort},   we have that \eqref{ep10'}  implies that 
$\mu_{z,x,t}$ does not depend on $z$, that is, for $\mm$-a.e.\ $z\in\KK$, $\la\mu_{z,x,t},\theta(\cdot)\ra=\la\mu_{x,t},\theta(\cdot)\ra:=\int_{\KK}\la\mu_{z,x,t},\theta(\cdot)\ra\,d\mm(z)$, for any $\theta\in C(K')$,
a.e.\ $(x,t)\in \X\Om\X(0,\infty)$.

3.  The central strategy of  the proof is then to show that $\mu_{x,t}=\d_{\xi(x,t)}$, with $\xi(x,t):=\bar f(x,\bar u(x,t))$, where  $\bar u(x,t)$ is the entropy solution of \eqref{p4}. In order to achieve this, a major step is to obtain the inequality
\begin{equation}\label{eEND}
\int_{Q}\langle {\mu}_{x,t},I\left(\cdot,\bar f(x,\bar u(x,t))\right)\rangle \varphi_t
 +\langle{\mu}_{x,t},G\left(\cdot,\bar f(x,\bar u (x,t))\right)\rangle\Delta \varphi \,dx\,dt \ge 0,
\end{equation}
for all $0\le \varphi \in C_c^{\infty}(Q)$, where
\begin{align}
&I(x,\rho,\alpha):= \int_{{\mathcal K}}|g(x,z,\rho)-g(x,z,\alpha)|\,d{\mathfrak m}(z),\label{p15}\\
& G(\rho,\alpha):= |\rho-\alpha|.\label{p16}
\end{align}
The inequality  \eqref{eEND} is obtained as follows. We first use \eqref{eTA3'}, in item (ii) of Theorem~\ref{T:A3},  making $u_1(x,t)=u_{\ve}(x,t)$ and $u_2(x)=\Phi_\a(x,\frac{x}{\ve})$.
Then, we set $\a=\bar f(y,\bar u(y,s))$, integrate in $(y,s)\in Q$, and make $\ve\to0$ to obtain, after some manipulations, 
\begin{align}\label{e9*}
&\int_{Q^2}\langle {\mu}_{x,t},I(x,\cdot,\xi(y,s))\rangle \phi_t
+\langle{\mu}_{x,t},G(\cdot,\xi(y,s))\rangle
\big({\Delta}_x\phi+{\div}_y{\nabla}_x\phi \big)\,dx\,dt\,dy\,ds\\
&=\lim_{\ve \to 0}\lim_{\delta \to 0}\int_{Q^2}
\bigg\{|{\nabla}_x[f(x,\frac{x}{\ve},u_{\ve})-f(x,\frac{x}{\ve},{\Phi}_{\xi}
(x,\frac{x}{\ve}))]|^2\nonumber\\
&\qquad+{\nabla}_y
[f(x,\frac{x}{\ve},u_{\ve})-f(x,\frac{x}{\ve},{\Phi}_{\xi}(x,\frac{x}{\ve}))]
\cdot{\nabla}_x[f(x,\frac{x}{\ve},u_{\ve})-f(x,\frac{x}{\ve},{\Phi}_{\xi}(x,\frac{x}
{\ve}))]\bigg\}\nonumber\\
&\qquad\qquad \times 
H_{\delta}'(f(x,\frac{x}{\ve},u_{\ve})-f(x,\frac{x}{\ve},{\Phi}_{\xi}
(x,\frac{x}{\ve})))\phi\,dx\,dt\,dy\,ds,\nonumber
\end{align}
where $u_{\ve}=u_{\ve}(x,t)$, $\xi=\xi(y,s)$.  

Next we  use again \eqref{eTA3'}, in item (ii) of Theorem~\ref{T:A3}, in variables $(y,s)\in Q$,  making $u_1(y,s)=\bar u(y,s)$ and $u_2(y)=\bar g(y,k)$, where $k\in\R$, and $\bar g$ is defined in \eqref{egbar}, to obtain
 \begin{align}\label{e6*}
&\int_{Q}|\bar g(y,k)-\bar{u}(y,s)|\phi_s +\sgn(k-\bar{f}(y,\bar{u}(y,s)))
{\nabla}_y\bar{f}(y,\bar{u})\cdot {\nabla}_y\phi \,dy\,ds \\
&\qquad\qquad=\lim_{\delta \to 0}\int_{Q}|{\nabla}_y\bar{f}(y,\bar{u})
|^2H_{\delta}'(k-\bar{f}(y,\bar{u}(y,s)))\phi\, dy\,ds,\qquad \text{for all $k\in\R$}.\nonumber
\end{align}
Precisely at this point we will need the additional restriction in \eqref{em0}, in the case where $f_\ve$ is of type 2.  Namely, we need the validity of the formula
 \begin{equation}\label{eval}
 \bar{u}(y,s)=\int_{{\mathcal K}}g(y,z,\xi(y,s))\,d{\mathfrak m}(z),
 \end{equation}
 which is guaranteed by \eqref{em0}, as proved in Lemma~\ref{L:6.2}. 
 Thus, using the definition of $I$ and $G$, the fact that, since ${\nabla}_y\xi(y,s)={\nabla}_y[f(x,\frac{x}{\ve},{\Phi}_{\xi(y,s)}
(x,\frac{x}{\ve}))]$, we have
\begin{align*}
&\int_{Q}|{\nabla}_y\bar{f}(y,\bar{u})|^2H_{\delta}'(k-
\bar{f}(y,\bar{u}(y,s)))\phi \,dy\,ds =\int_{Q}|{\nabla}_y\xi(y,s)|^2H_{\delta}'(k-\xi(y,s))\phi \,dy\,ds\\
&=\int_{Q}|{\nabla}_yf(x,\frac{x}{\ve},{\Phi}_{\xi(y,s)}(x,\frac{x}{\ve}))|^2
H_{\delta}'(k-\xi(y,s))\phi \,dy\,ds,
\end{align*}
we arrive at 
\begin{equation*}
\int_{Q}I(y,k,\xi(y,s))\phi_s +G(k,\xi(y,s)) {\Delta}_y\phi \,dy\,ds =\lim_{\delta \to 0}\int_{Q}|{\nabla}_yf(x,\frac{x}{\ve},{\Phi}_{\xi(y,s)}
(x,\frac{x}{\ve}))|^2
H_{\delta}'(k-\xi(y,s))\phi \,dy\,ds.
\end{equation*}
for all $k\in \re$ and all $0\le \phi \in C_c^{\infty}(Q^2)$.

We then  take $k=f(x,\frac{x}{\ve},u_{\ve}(x,t))$, integrate in $(x,t)$,   make $\ve\to0$, using Theorem~\ref{T3}, use the definition of $\mu$, and after some manipulations we obtain  
\begin{align}\label{e10*}
&\int_{Q^2}\langle {\mu}_{x,t},I(y,\cdot,\xi(y,s))\rangle \phi_s
+\langle{\mu}_{x,t},G(\cdot,\xi(y,s))\rangle
\big({\Delta}_y\phi+{\div}_x{\nabla}_y\phi\big)\,dx\,dt\,dy\,ds\\
&=\lim_{\ve \to 0}\lim_{\delta \to 0}\int_{Q^2}\bigg\{
|{\nabla}_y[f(x,\frac{x}{\ve},u_{\ve})-f(x,\frac{x}{\ve},{\Phi}_{\xi}
(x,\frac{x}{\ve}))]|^2\nonumber\\
&+ {\nabla}_y[f(x,\frac{x}{\ve},u_{\ve})-f(x,\frac{x}{\ve},{\Phi}_{\xi}
(x,\frac{x}{\ve}))]\cdot{\nabla}_x[f(x,\frac{x}{\ve},u_{\ve})
-f(x,\frac{x}{\ve},{\Phi}_{\xi}(\frac{x}
{\ve}))]\bigg\}\nonumber\\
&\qquad\qquad \times
H_{\delta}'(f(x,\frac{x}{\ve},u_{\ve})-f(x,\frac{x}{\ve},{\Phi}_{\xi}
(x,\frac{x}{\ve}))\phi\,dx\,dt\,dy\,ds.\nonumber
\end{align}

Finally we add  \eqref{e9*} with \eqref{e10*}, use suitable test functions as in  Kruzkov's doubling variables  method  (cf.\ \cite{Kr}), as in the proof of Theorem~\ref{T:A3}, to  conclude the proof of  \eqref{eEND}.

4. From \eqref{eEND} it follows  that $\mu_{x,t}=\d_{\xi(x,t)}$, with $\xi(x,t)=\bar f(x,\bar u(x,t))$, as asserted.  This is achieved in a standard way, where an essential  point is to show that 
\begin{equation}\label{eEND4}
\lim_{\tau\to0}\frac1{\tau}\int_0^\tau \int_{B_R}\la\mu_{x,t},I(x,\cdot,\bar f(x,\bar u_0(x)))\ra\,dx\,dt=0,\qquad \text{for all $R>0$},
\end{equation}
where $B_R$ is the open ball centered at the origin with radius $R$.
It is in the proof of \eqref{eEND4} that we need to use the fact that $u_0(z,x)$ has the for $u_0(z,x)=g(x,z,\phi_0(x))$ in hypotheses {\bf(h1.2)}, if $f_\ve$ is of type 1, or  {\bf(h2.2)}, if 
$f_\ve$ is of type 2.  Indeed, \eqref{eEND4} follows from the relation 
\begin{align}\label{p14}
&\int_{Q}\int_{{\mathcal K}}\langle {\nu}_{z,x,t},|\cdot-
{\Phi}_{\alpha}(x,z)|\rangle \varphi_t+\langle {\nu}_{z,x,t},|f(x,z,\cdot)-
f(x,z,{\Phi}_{\alpha}(x,z))|\rangle \Delta\varphi(x,z) \,d{\mathfrak m}(z)\,dx\,dt \\
&\qquad\qquad\qquad+\int_{\Om}\int_{{\mathcal K}}|u_{0}(x,z)-
{\Phi}_{\alpha}(x,z)|\varphi(x,0)\, d{\mathfrak m}(z)\,dx \ge 0\nonumber
\end{align}
for all $\alpha \in \re$ and for all $0\le \varphi \in  C_c^{\infty}(Q)$, obtained from \eqref{p8} by sending $\ve\to0$ and using Theorem~\ref{T3}.  From \eqref{p14}, using the definition of $\mu_{x,t}$, we obtain 
 \begin{align}\label{p19}
&\int_{Q}\langle {\mu}_{x,t},I(x,\cdot,\alpha)\rangle
\varphi_t + \langle {\mu}_{x,t},G(\cdot,\alpha)\rangle 
\Delta \varphi \,dx\,dt\\
&\qquad\qquad\qquad+\int_{\Om}\int_{{\mathcal K}}|u_{0}(x,z)-
{\Phi}_{\alpha}(x,z)|\varphi(x,0)\,d{\mathfrak m}(z)\,dx\ge 0,\nonumber
\end{align}
for all $0\le \varphi \in C_{c}^{\infty}(Q)$ and all 
$\alpha \in \re $. Now, from \eqref{p19} in a standard way, we obtain 
\begin{equation}\label{p20}
\overline{\lim_{h\to 0}}\frac{1}{h}\int_{0}^h\int_\Om\langle 
{\mu}_{x,t},I(x,\cdot,\alpha)\rangle \phi(x) \,dx\,dt\le 
\int_{\Om}\int_{{\mathcal K}}|u_{0}(z,x)-{\Phi}_{\alpha}(x,z)|
\phi(x) \,d{\mathfrak m}(z)\,dx,
\end{equation}
for all $0\le \phi\in C_0^\infty(\Om)$. 
Using the flexibility provided by the presence of the test function $\phi$ in \eqref{p20}, we get to replace $\a$ by $\phi_0(x)$ in \eqref{p20},  then getting \eqref{eEND4}. 

5. Therefore, using the definition of $\mu_{x,t}$, 
we deduce that $\nu_{z,x,t}=\d_{g(x,z,\bar f(x,\bar u(x,t)))}$, and so by Theorem~\ref{T3}
$$
\lim_{\ve\to0} \int_{Q} u_\ve(x,t)\phi(x,t)\,dx\,dt=\int_{Q}\int_{\KK}g(x,z,\bar f(x,\bar u(x,t)))\phi(x,t)\,d\mm(z)\,dx\,dt=\int_{Q}\bar u(x,t)\phi(x,t)\,dx\,dt.
$$
Finally, in the case where $f_\ve$ is of type~1, for all $\ve>0$, using Lemma~\ref{L:E},  or, in the case where $f_\ve$ is of type~2, for all $\ve>0$, using Lemma~\ref{T:3.2},  we obtain \eqref{p6}, which finishes the proof.

\end{proof}

\section{Homogenization of Porous medium type equations: \\  Bounded domains, regular algebras w.m.v.\  and general initial data}\label{S:6}

In this section we address the same homogenization problem as in the last section, but here we drop the restriction that  the initial data should be well-prepared, allowing a general initial data.  However, we have to compromise and restrict ourselves to bounded domains. Besides, the method used in the homogenization analysis here, which completely differs from the technique used in the last section, only allows us to consider ergodic algebras which are regular algebras w.m.v. We will also need a further mild restriction on the pressure function $f$ in order to obtain  the corresponding corrector property. Namely, we will need to assume the strict convexity of $\bar G(x,\cdot)$, for all $x\in\Om$, where $\bar G$ is defined in \eqref{eGbar} below.

So, in this section we only assume the following on the initial data:
\begin{equation}\label{e7.1}
u_0(x,z)\in L^\infty(\Om;\AA(\R^n)).
\end{equation}

 We will use the concept and some basic facts about viscosity solutions of fully nonlinear parabolic equations. 
 We refer to \cite{CIL} for a general exposition of the theory of viscosity solutions of fully-nonlinear elliptic and parabolic equations. 
 
 Before stating the theorem, let us introduce the following notations. Given a function $h\in L^\infty (\Om)$, we denote by $\D^{-1} h$ the solution of the boundary value problem
 \begin{equation}\label{eDinv}
 \begin{cases}\D v(x)=h(x), & x\in \Om,\\
                v(x)=0, &x\in\po \Om.
                \end{cases}
                \end{equation}

\begin{theorem}\label{T:7.1} Consider the problem \eqref{p1}. Assume that, for each fixed $\ve>0$, $f_\ve(x,u)$ is either of type~1 and satisfies {\bf(h1.1)}  or it is of type~2 and satisfies {\bf(h2.1)},  where $\AA(\R^n)$ is a regular algebra w.m.v., while $u_0(x,z)$ satisfies \eqref{e7.1}.
Let $u_\ve(x,t)$ be the entropy solution of \eqref{p1}. 
Then, as $\ve\to0$, $u_\ve$ weak star converges in $L^\infty(\Om\X[0,\infty))$ to the entropy solution, $\bar u(x,t)$, of the problem~\eqref{p4}.
Moreover,  we have  
\begin{equation}\label{crt}
u_\ve(x,t)-  g\bigl(x,\frac{x}{\ve},\bar f(x,\bar u(x,t))\bigr) \to 0 \qquad\text{as $\ve\to0$ in $L_\loc^1(\Om\X[0,\infty))$}.
\end{equation}
\end{theorem}

\begin{proof}
1.  The fact that the solutions of \eqref{p1} form a uniformly bounded sequence in
$L^\infty(Q)$, was established in the proof of Theorem~\ref{T:6.1}.  

2. Now, let us make a general observation concerning problem \eqref{e01}-\eqref{e01''}, under assumptions {\bf(f1.1)}-{\bf(f1.2)}, or {\bf(f2.1)}-{\bf(f2.3)}, for $f$ of type~1 or type~2, respectively. So,  let $u$ be the entropy solution of~\eqref{e01}--\eqref{e01''} and, for each  $t\in[0,\infty)$, let $U(\cdot,t):=\D^{-1}u(\cdot,t)$. 
We claim that $U$ is the viscosity solution of
\begin{equation} \label{La2}
\begin{cases}
{\partial}_tU-f(x,\Delta U)=0, &(x,t)\in Q, \\ 
U(x,0)=U_{0}(x), &x\in \Om,
\\ 
U(x,t)=0,& (x,t)\in\po\Om\X(0,\infty),
\end{cases}
\end{equation}
where $U_{0}=\D^{-1}u_0$. 
Indeed, let $u_{\s}$ be the smooth solution of the corresponding regularized problem \eqref{e04}--\eqref{e04''}. For each $t\in[0,\infty)$, let $U_{\s}(\cdot,t):=\D^{-1}u_\s(\cdot,t)$. 
Since $u_\s$ and $U_{\s}$ are smooth, it is clear that the latter is the (viscosity) solution of 
\begin{equation} \label{La4}
\begin{cases}
{\partial}_t U -f^\s(x,\Delta U)=0,
&(x,t)\in Q,\\ 
U(x,0)=U_{0,\s}(x), &x\in\Om,
\\ 
U(x,t)=0, &(x,t)\in\po\Omega\X(0,\infty),
\end{cases}
\end{equation}
where $U_{0,\s}=\D^{-1}u_{0,\s}$.
Since $\{u_\s(x,t)\}_{0<\s<1}$ is uniformly bounded in $L^\infty(\Om\X[0,\infty))$,  we easily see that the $U_\s(x,t)$ form a uniformly bounded sequence in
$L^\infty([0,\infty);W^{2,p}(\Om))$ for all $p\in(1,\infty)$. On the other 
hand, from~\eqref{La4},  we easily deduce that 
$|U_{\s}(x,t)-U_{\s}(x,s)|\le C|t-s|$ for all $x\in\Omega$ for 
some constant $C>0$, independent of $\s$. Hence, we see that 
$U_{\s}$ is uniformly bounded in 
$W^{1,\infty}(\bar{Q})$. In particular, there is a 
subsequence $U_{\s_i}$ of $U_{\s}$ converging locally uniformly 
in $\bar{Q}$ to a function 
$U\in W^{1,\infty}(\bar{Q})$ which satisfies $U=\D^{-1}u$. 

It follows in a standard way that $U$ is the viscosity solution of \eqref{La2}. Indeed, given any $(x_0,t_0)\in Q$, we consider $\varphi\in C^2(Q)$ such that $U-\varphi$ has a strict local maximum at $(x_0,t_0)$. Since $U_{\s_i}-\varphi$ converges locally uniformly in 
$\bar{Q}$ to the function $U-\varphi$, we may obtain a sequence $(x_{i},t_i)\in Q$ such that $(x_i,t_i)$ is a  point of local maximum of 
$U_{\s_i}-\varphi$ and 
$(x_{i},t_i)\to (x_0,t_0)$ as $i\to\infty$.  Thus, we have
$$
\partial_t\varphi(x_i,t_i)-f^{\s_i}(x_i,\Delta \varphi(x_i,t_i))\le 0,
$$
from which it follows, as $i\to\infty$, 
\begin{equation}\label{evisc}
\partial_t\varphi(x_0,t_0)-f(x_0,\Delta \varphi(x_0,t_0))\le 0.
\end{equation}
To relax the assumption of a strict local maximum to just a local maximum we proceed as usual replacing $\varphi$ by, say, $\tilde\varphi(x,t):=\varphi(x,t)+\d(|x-x_0|^2+(t-t_0)^2)$ obtaining \eqref{evisc} with $\tilde\varphi$ instead of $\varphi$ and from that we obtain again \eqref{evisc} for $\varphi$ passing to the limit when $\d\to0$. In an entirely similar way we prove the reverse inequality when $U-\varphi$ has a local minimum at $(x_0,t_0)$, so proving that $U$ is a viscosity solution of \eqref{La2}.

3. In this and the next step we shall study the homogenization of~\eqref{e2.20} using a method 
motivated by~\cite{I}. As we will see,  the $\ve$-Laplacian property in Lemma~\ref{L:1.3} plays a decisive role at this point, and this explains our assumption that $\AA(\R^n)$ is a regular algebra w.m.v. 
We define $U_\ve(x,t)$ in $\Om\X[0,\infty)$ by $U_\ve:=\D^{-1}u_\ve$ 
where $u_\ve$ is the entropy solution of~\eqref{p1}. By step 2, we have that 
$U_\ve$ is the viscosity solution of
\begin{equation} \label{e2.20}
\begin{cases}{\partial}_tU_\ve-f(x,\frac{x}{\ve},\Delta U_\ve)=0,
&(x,t)\in Q,
\\ 
U_\ve(x,0)=U_{0,\ve}(x),
&x\in \Om,
\\ 
U_\ve(x,t)=0,& (x,t)\in\po\Om\X(0,\infty).
\end{cases}
\end{equation}
where $U_{0,\ve}=\D^{-1}u_{0,\ve}$, with $u_{0,\ve}(x)=u_0(x,\frac{x}{\ve})$. 
 The same argument used in the previous step shows that 
$$
 U_\ve\in L^\infty((0,\infty);W^{2,p}(\Om))\bigcap\Lip((0,\infty);L^\infty(\Om)),
$$
and so there is a subsequence $U_{\ve_i}$ of $U_\ve$ converging locally uniformly in $\bar Q$ to a function 
$$
\bar U\in L^\infty((0,\infty);W^{2,p}(\Om))\bigcap\Lip((0,\infty);L^\infty(\Om)),
$$
in particular, $\bar U\in W^{1,\infty}(\bar Q)$.

4. We claim that $\bar U(x,t)$ is the viscosity solution of the initial-boundary value problem
\begin{equation} \label{e2.23}
\begin{cases}
{\partial}_tU-\bar{f}(x,\Delta U)=0,
&(x,t)\in Q,
\\ 
U(x,0)=\bar{U}_{0}(x),
&x\in \Om,
\\ 
U(x,t)=0,& (x,t)\in\po\Om\X(0,\infty).
\end{cases}
\end{equation}
where 
$$
\bar U_0:=\D^{-1} \Medint u_0(z,x)\,dz.
$$ 

 Indeed, let $(\hat x,\hat t)\in Q$ and let $\varphi\in C^2(Q)$ be such  $\bar U - \varphi$ has a local maximum at $(\hat x,\hat t)$. Also,  let $v_{\s,\d}\in\AA(\R^n)$ be a smooth function satisfying
\begin{equation}
  g_\s\left(\hat x, z,\bar f_\s(\hat x,p)\right)-p -\d\le\gD_z v_{\s,\d}\le g_\s\left(\hat x,z,\bar f_\s(\hat x,p)\right)-p
+\d,  \label{e2.28}
\end{equation}
 with $p=\gD\varphi(\hat x,\hat t)$, whose existence is asserted by Lemma~\ref{L:1.3}, where $g_\s(x,y,\cdot)$ is the inverse of $f_\s(x,y,\cdot)=f(x,y,\cdot)+\s \cdot$, and $\bar f_\s$ is given by \eqref{egbar}, \eqref{efbar} with $g_\s$ replacing $g$. In particular, given any $\d'>0$ we can find $\d>0$ sufficiently small such that
\begin{equation*}
\bar f_\s(\hat x,\gD \varphi(\hat x,\hat t))-\d'\le f_\s\left(\hat x,z,\gD\varphi(\hat x,\hat t)+\gD v_{\s,\d}(z)\right)\le \bar f_\s(\hat x,\gD \varphi(\hat x,\hat t))+\d',
\end{equation*}
from which it follows
\begin{equation*}
\bar f(\hat x,\gD \varphi(\hat x,\hat t))-2\d'\le f\left(\hat x,z,\gD\varphi(\hat x,\hat t)+\gD v_{\d}(z)\right)\le \bar f(\hat x,\gD \varphi(\hat x,\hat t))+2\d',
\end{equation*}
for $\s>0$ sufficiently small, since $\bar f_\s$ converges pointwise to $\bar f$, where we set $v_\d:=v_{\s,\d}$. Here we use {\bf(h2.4)}, in case $f$ is of type~2, which implies that $\bar f$ is strictly increasing in 
$\bar \Om\X I$ and $\bar u(x,t)$ assumes values in $I$, so that we may assume that $\bar f(x,\cdot)$ is strictly increasing in $\R$, and so $\bar f_\s$ converges everywhere to 
$\bar f$.

Take $\rho>0$ be small enough, and let $(x_j,t_j)\in Q$, be a  point of maximum of
$$
U_j(x, t)-\varphi(x, t)-\ve_j^2 v_\d(\frac{x}{\ve_j})-\rho(|x-\hat x|^2+(\hat t-t)^2),
$$
 where we denote $U_j=U_{\ve_j}$, such that  $(x_j,t_j)\to(\hat x,\hat t)$, as $j\to\infty$. Such sequence $(x_j,t_j)$ exists  since $U_j$ converges locally uniformly to $\bar U$ and $v_\d$ is bounded. We have
$$
\varphi_{t}(x_j,t_j)-f\left(x_j,\frac{x_j}{\ve_j},\gD \varphi(x_j,t_j)+\gD v_\d(\frac{x_j}{\ve_j})+\rho\right)\le
2\rho|\hat t-t_j|,
$$
and
$$
f\left(\hat x,\frac{x_j}{\ve_j},\gD\varphi(\hat x,\hat t)+\gD v_\d(\frac{x_j}{\ve_j})\right)\le\bar f(\hat x,\gD \varphi(\hat x,\hat t))+2\d',
$$
which, after addition, gives
$$
\varphi_{t}(\hat x,\hat t)-\bar f(\hat x,\gD\varphi(\hat x,\hat t))\le 
O(|x_j-\hat x|+|\hat t-t_j|)+O(\rho)+2\d'.
$$
Hence, letting $j\to\infty$ first, and then letting $\rho,\d'\to0$, we obtain
$$
\varphi_t(\hat x,\hat t)-\bar f(\hat x,\gD \varphi(\hat x,\hat t))\le 0.
$$
The reverse inequality, when $\bar U - \varphi$ has a local minimum at $(\hat x,\hat t)$, follows in an entirely similar way, which concludes the proof of the claim.

5. By the uniqueness of the viscosity solution 
of~\eqref{e2.23} (see for instance~\cite{CIL}, Theorem~8.2), we conclude that the whole sequence $U_\ve(x,t)$ converges locally uniformly to $\bar U(x,t)$. Let 
$\bar{u}:=\D \bar U$.  Given any $\varphi\in C^{\infty}_c(Q)$, we have
\begin{align*}
&\int_{Q}u_\ve(x,t)\varphi(x,t)\,dx\,dt=
\int_{Q}\Delta U_\ve\varphi\,dx\,dt=\int_{Q}U_\ve 
\Delta\varphi\,dx\,dt\stackrel{\ve\to 0}{\longrightarrow}\\
&\qquad\qquad\int_{Q}\bar{U}\Delta \varphi\,dx\,dt=
\int_{Q}\bar{u}\varphi\,dx\,dt. 
\end{align*}
Consequently, $u_\ve(x,t)$  converges in the weak star topology of $L^\infty(Q)$ to $\bar u=\gD\bar U(x,t)$. 
Now, let  $\tilde u$ be the entropy solution of  \eqref{p4}. Let $\tilde U:=\D^{-1}\tilde u$. As it was done above, we easily prove that $\tilde U$ is the viscosity solution of \eqref{e2.23}. Therefore, $\tilde U\equiv \bar U$, and so $\tilde u=\bar u$. This proves the first assertion in the statement of the theorem.

6. Now, we observe that, for each $\ve>0$,  the identity
\begin{equation}\label{eUe}
\partial_tU_\ve-f(x,\frac{x}{\ve},\gD U_\ve)=0,
\end{equation}
holds in the sense of distributions in $Q$. Indeed, for any $\varphi\in C_0^\infty((0,\infty);H_0^1(\Om))$, we have
\begin{equation}\label{e1102}
\int_Q u_\ve\varphi_t-\nabla f(x,\frac{x}{\ve},u_\ve)\cdot\nabla\varphi\,dx\,dt=0.
\end{equation}
Given $\phi\in C_0^\infty(Q)$, we take $\varphi=\D^{-1}\phi$ in \eqref{e1102}, use $u_\ve=\D U_\ve$ and integration by parts, to obtain that
\begin{equation}\label{e1102'}
\int_Q U_\ve\phi_t+f(x,\frac{x}{\ve},\D U_\ve)\phi\,dx\,dt=0,
\end{equation}  
holds for any $\phi\in C_0^\infty(Q)$. Similarly, since $\bar u$ is the entropy solution of \eqref{p4}, we have
\begin{equation}\label{e1102''}
\int_Q \bar U\phi_t+\bar f(x,\D \bar U)\phi\,dx\,dt=0,
\end{equation}  
for any $\phi\in C_0^\infty(Q)$. In particular, for any $\phi\in C_0^\infty(Q)$,  we have
\begin{equation}\label{e120}
\begin{aligned}
\lim_{\ve\to0}\int_Q f(x,\frac{x}{\ve}, u_\ve(x,t)) \phi(x,t)\,dx\,dt&=-\lim_{\ve\to0}\int_Q U_\ve(x,t)\phi_t(x,t)\,dx\,dt\\
&=-\int_Q \bar U(x,t)\phi_t(x,t)\,dx\,dt\\
& =-\int_{Q}\bar f(x,\bar u(x,t))\phi(x,t)\,dx\,dt,
\end{aligned}
\end{equation}
so that $v_\ve(x,t):=f(x,\frac{x}{\ve}, u_\ve(x,t))$ weak star converges in $L^\infty(Q)$ to $\bar v(x,t):=\bar f(x,\bar u(x,t))$.

7. By applying Theorem~\ref{T:T2}, we may obtain a subnet of $u_\ve$, which we will still denote by $u_\ve$, and a (weakly measurable) parameterized family of probability measures on a compact interval of $\R$, $\{\nu_{x,z,t}\}$, $(x,z,t)\in\Om\X\KK\X(0,\infty)$,  which form a so called family of two-scale Young measures.  As in the previous section, let us consider the following parametrized family of probability measures $\{\mu_{x,z,t}\}$, $(x,z,t)\in\Om\X\KK\X(0,\infty)$, defined by
\begin{equation}\label{e121}
\la \mu_{x,z,t}, \zeta(\cdot)\ra =\la \nu_{x,z.t}, \zeta (f(x,z,\cdot))\ra, \quad \zeta\in C(\R).
\end{equation}
We claim that $\mu_{x,z,t}=\d_{\bar f(x,\bar u(x,t))}$ for a.e.\ $(x,t)\in Q$ and $\mm$-a.e.\ $z\in\KK$.

Indeed, let us  introduce the function $G_*:\bar\Om\X\R^n\X\R\to\R$ defined by
\begin{equation}\label{eG}
G_*(x,z,v):=\int_0^v g(x,z,s)\,ds,
\end{equation}
where $g$ is defined in {\bf(h1.1)}, in the case where $f_\ve$ is of type~1, for all $\ve>0$, or in {\bf(h2.1)}, when $f_\ve$ is of type~2. 
We can easily verify that $G(x,v,\cdot)\in\AA(\R^n)$. We also define $\bar G:\bar\Om\X\R\to\R$ by
\begin{equation}\label{eGbar}
\bar G_*(x,v):=\Medint_{\R^n} G_*(x,z,v)\,dz.
\end{equation}
The function $G_*(x,z,v)$ satisfies an uniform strict convexity condition, in the sense that, for $0<\theta<1$ and $v_1<v_2$, we have
\begin{equation}\label{eG2}
(1-\theta) G_*(x,z,v_1)+\theta G_*(x,z,v_2)-G_*(x,z,(1-\theta)v_1+\theta v_2)\ge C\theta(1-\theta)(v_2-v_1)^2,
\end{equation}
where $C>0$ is such that
$$
g(x,z,v_2)-g(x,z,v_1)\ge C(v_2-v_1),
$$
uniformly with respect to $(x,z)\in\bar\Om\X\R^n$, which can be easily verified.

We now  begin by using an argument by Visintin in theorem~2.1 of \cite{Vi}. So, we first observe that, for any $\ve>0$, $u_\ve(x,t)\in \po G_*(x,\frac{x}{\ve},v_\ve(x,t))$,
where $\po G_*(x,z,\cdot)$ denotes the subdifferential of the convex function $G_*(x,z,\cdot)$ defined by \eqref{eG}, which easily follows form the definition of $v_\ve(x,t)$. We also observe that $v_\ve$ is uniformly bounded in $L^2((0,T); H_0^1(\Om))$, for all $T>0$ (see the estimate in step 3 of Theorem~\ref{t01}).  On the other hand, since $\Om$ is bounded,  from \eqref{p1}, $u_\ve$ is uniformly bounded in $L^2((0,T);L^2(\Om))$ and in $H^1((0,T);H^{-1}(\Om))$,  for all $T>0$, and so, by Aubin lemma (see, e.g., \cite{LJ}), $u_\ve$ strongly converges to 
$\bar u$ in $L^2((0,T); H^{-1}(\Om))$. Hence, from the relation
\begin{equation}\label{e122}
G_*(x,\frac{x}{\ve},v_\ve(x,t))-G_*(x,\frac{x}{\ve}, \bar v(x,t))\le u_\ve(x,t)(v_\ve(x,t)-\bar v(x,t)),
\end{equation}
which follows from the convexity of $G_*(x,z,\cdot)$ and the fact that $u_\ve(x,t)\in\po G_*(x,\frac{x}{\ve}, v_\ve(x,t))$, it follows by Theorem~\ref{T:T2} that, for all $\phi\in C_0^\infty(Q)$, we have
\begin{equation}\label{e123}
\int_{Q}\int_{\KK}\la \mu_{x,z,t}, G_*(x,z, \cdot)\ra \phi(x,t)\,d\mm(z)\,dx\,dt\le \int_Q \bar G_*(x,\bar v(x,t))\phi(x,t)\,dx\,dt.
\end{equation}

Now, since $v_\ve(x,t)$ is uniformly bounded in $L^2((0,T);H_0^1(\Om))$, for all $T>0$, for any $\zeta\in C^1(\R)$, $\phi\in C_0^\infty(Q)$, and $\varphi\in\AA(\R^n)$, such that $\po_{z_i}\varphi, \po_{z_iz_j}^2\varphi\in\AA(\R^n)$, $i,j=1,\cdots,n$,  we have
\begin{equation}\label{e124}
\begin{aligned}
0&=\lim_{\ve\to0} \ve^2 \int_{Q}\nabla_x\zeta(v_\ve(x,t))\cdot\nabla_x (\phi(x,t)\varphi(\frac{x}{\ve}))\,dx\,dt\\
 &= \int_{Q}\int_{\KK}\la \mu_{x,z,t},\zeta(\cdot)\ra \D_z\varphi(z) \phi(x,t)\,d\mm(z)\,dx\,dt.
 \end{aligned}
 \end{equation}
 Hence, as in the proof of Theorem~\ref{T:6.1}, using Lemma~\ref{L:ort}, we conclude that $\mu_{x,z,t}=\bar \mu_{x,t}:=\int_{\KK}\mu_{x,z,t}\,d\mm(z)$, for a.e.\ $(x,t)\in Q$ and $\mm$-a.e.\ $z\in\KK$. But then, by \eqref{e123}, we have
 \begin{equation}\label{e125}
 \int_Q \la \bar \mu_{x,t}, \bar G_*(x,\cdot)\ra \phi(x,t)\,dx\,dt\le \int_Q \bar G_*(x,\bar v(x,t))\phi(x,t)\,dx\,dt.
 \end{equation}
 But, since $\bar v(x,t)=\int_{\R}\l\,d\bar\mu_{x,t}(\l)$, for a.e.\ $(x,t)\in Q$, from Jensen inequality it follows that 
 \begin{equation}\label{e126}
 \int_Q \la \bar \mu_{x,t}, \bar G_*(x,\cdot)\ra \phi(x,t)\,dx\,dt\ge \int_Q \bar G_*(x,\bar v(x,t))\phi(x,t)\,dx\,dt,
 \end{equation} 
 and so we have the equality
\begin{equation}\label{e127}
 \int_Q \la \bar \mu_{x,t}, \bar G_*(x,\cdot)\ra \phi(x,t)\,dx\,dt= \int_Q \bar G_*(x,\bar v(x,t))\phi(x,t)\,dx\,dt,
 \end{equation}  
which, from the strict convexity of $\bar G_*(x,\cdot)$ (see \eqref{eG2}), for $x\in Q$, implies that 
\begin{equation}\label{e128}
\mu_{x,z,t}=\d_{\bar v(x,t)},\quad \text{for a.e.\ $(x,t)\in Q$ and $\mm$-a.e.\ $z\in\KK$}.
\end{equation}

8. Now, using the definition of $\mu_{x,z,t}$, and the fact that, for $\mm$-a.e.\ $z\in\KK$ and a.e.\ $x\in\Om$, $g(x,z,f(x,z,u))=u$, for all $u$,  where for $f_\ve$ of type 2 we use \eqref{em0} in  {\bf (h2.1)}, we arrive at
\begin{equation}\label{e129}
\nu_{x,z,t}=\d_{g(x,z,\bar f(x,u(x,t)))},\quad \text{for $\mm$-a.e.\ $z\in\KK$, and a.e.\ $(x,t)\in Q$}.
\end{equation}

Finally, in the case where $f_\ve$ is of type~1, for all $\ve>0$, using Lemma~\ref{L:E},  or, in the case where $f_\ve$ is of type~2, for all $\ve>0$, using Lemma~\ref{T:3.2},  we obtain \eqref{crt}, which finishes the proof.

\end{proof}

\section*{Acknowledgements}

H.~Frid gratefully acknowledges the support of CNPq, grant 306137/2006-2 ,
and FAPERJ, grant E-26/152.192-2002.

\end{document}